\documentclass[11pt, a4paper]{article}

\usepackage{amsmath}

\usepackage{amsmath, amssymb, amsthm,color}
  \newtheorem{theorem}{Theorem}[section]
 \newtheorem{definition}{Definition}[section]
  
    \newtheorem{lemma}{Lemma}[section]
    \newtheorem{proposition}{Proposition}[section]
\newtheorem{rem}{Remark}[section]

\newcommand{\set}[2]{\left\{#1\,:\,#2\right\}}
\usepackage{dsfont}

  \def\Proof{\noindent {\bf Proof : }$\ $\\ }

\numberwithin{equation}{section}

\newcommand{\Z}{\mathbb{Z}}

\newcommand{\R}{\mathbb{R}}

\newcommand{\Supp}{\rm Supp }

\newcommand{\bq}{\begin{equation}}
\newcommand{\eq}{\end{equation}}

\begin{document}

\title{Time-periodic forcing and asymptotic stability for the Navier-Stokes-Maxwell  equations}
\author{Slim Ibrahim\footnote{Department of Mathematics and Statistics, University of Victoria; e-mail: ibrahims@uvic.ca}, Nader Masmoudi\footnote{Courant Institute, New-York University; e-mail: masmoudi@courant.nyu.edu}, and Pierre Gilles Lemari\'e--Rieusset\footnote{Laboratoire de Math\'ematiques et Mod\'elisation d'\'Evry, UMR CNRS 9071, UEVE, ENSIIE; e-mail: plemarie@univ-evry.fr}}

\date{}\maketitle

\begin{abstract}   
For the 3D Navier-Stokes-Maxwell problem on the whole space and in the presence of external time-periodic forces, first we study the existence of time-periodic small solutions, and then we prove their asymptotic stability. We use new type of spaces to account for averaged decay in time.

$\ $
\end{abstract}

\section*{Introduction.}
Physical and analytical models already exist for both electro-hydrodynamic and magneto-hydrodynamic. However quite often in actual situations, both combined electro and magneto-hydrodynamic effects occur. Recent works attempted to develop fully consistent, multi-dimensional, unsteady and incompressible flows of electrically conducting fluids under the simultaneous or separate influence of externally applied or internally generated electric and magnetic fields.  The approach is based on the use of fundamental laws of continuum mechanics and thermodynamics. See for example \cite{Ko-Dulikravich}. 
However, because of the considerable complexity of even simpler versions of the combined electro-magneto-hydrodynamic models, it is still hard to analyze, even numerically, the combined influence of electric and magnetic fields and the fluid flow. In this paper, we analyze an adiabatic situation where thermodynamical effects are neglected, and nonlinearities are reduced to the only action of Lorentz force.\\
More specifically, consider a three-dimensional incompressible, viscous and charged fluid with a velocity field $\vec u$. Fluid charged particles motion generates an electro-magnetic field $(\vec E,\vec B)$ satisfying Maxwell equations, and a current $\vec J$ that acts backs on the fluid through Lorentz force. We assume that the current is given by Ohm's law $J=\sigma(\vec E+\vec u\wedge\vec B)$. Thus, the Navier-Stokes-Maxwell system we study reads as 
\begin{equation}
\label{eqns}
\left\{\begin{split} \partial_t \vec u+\text{ div } (\vec u\otimes \vec u)&= \nu\Delta\vec u+\vec J\wedge \vec B+\vec F_{\rm per}-\vec\nabla p\\  \partial_t\vec E-\vec\nabla\wedge\vec B&=-\vec J+\vec G_{\rm per}
\\ \partial_t\vec B+\vec\nabla\wedge \vec E&=\vec H_{\rm per}
\\ \text{ div }\vec u= \text{ div }\vec  B&=0\\ \vec J&=\sigma(\vec E+\vec u\wedge\vec B).
\end{split}\right.
\end{equation}
Here, $\vec u,\;\vec E,\;\vec B:\mathbb R^+_t\times\mathbb R^3_x\longrightarrow \mathbb R^3$ are vector fields defined on $\mathbb R^3$, and  the scalar function  
$p$  stands for the pressure. 
The positive parameters $\nu$ and  $\sigma$ represent the viscosity of the fluid and the electric resistivity, respectively. The self-interaction force term $\vec J \times\vec B $ in the Navier-Stokes equations comes from the Lorentz force under a quasi-neutrality assumption of the net charge carried by the fluid. Notice that taking into account a moving reference frame of the fluid, yields the correction $u\times B$ to the classical Ohm's law and keeps Faraday's law invariant under Gallilean transformation. The second  equation in \eqref{eqns} is the Amp\`ere-Maxwell equation for the electric field $\vec E$. The third equation is nothing but Faraday's law.  
For a detailed introduction to similar models and the theory of MHD, we refer to Davidson \cite{Davidson01} and Biskamp \cite{Biskamp93}. 
In \eqref{eqns}, the external forces $\vec F_{\rm per}$,  $\vec  G_{\rm per}$ and  $\vec H_{\rm per}$ are taken time-periodic: for a fixed time period $T>0$, we have
$$
\vec  F_{\rm per}(t+T,x)=\vec F_{\rm per}(t,x), \  \vec  G_{\rm per}(t+T,x)=\vec G_{\rm per}(t,x), \ \vec  H_{\rm per}(t+T,x)=\vec H_{\rm per}(t,x). 
$$ 
Before going any further, let us emphasize that despite the possible non-physical full relevance of \eqref{eqns},  the system still captures the various mathematical challenges of the full  complicated original set of equations. Indeed, \eqref{eqns} is a coupling of a dissipative equation of parabolic type (Navier-Stokes) with a hyperbolic system (Maxwell's equations). Despite the presence of damping in Ohm's law, solutions to Maxwell's equations do not enjoy any smoothing effect, due to the hyperbolic nature of the equations. Moreover, as this will be seen, a few challenges also arise when dealing with different decay rates, for different linear parts, and also different frequency sizes caused by the coupling.\\

Note that the pressure $p$ can still be determined using Leray projection from $\vec u$ and $\vec J\wedge\vec B$ via an explicit 
Cald\'eron-Zygmund operator (see \cite{Chemin95} for instance):
$$ \vec\nabla p=\vec\nabla\frac{1}{\Delta}\text{ div } (\vec J\wedge\vec B+\vec F_{\rm per}- \text{ div } (\vec u\otimes \vec u)).$$
In all what follows, we denote the solution to \eqref{eqns} by
$$
\vec\Gamma:=(\vec U,\vec E,\vec B).
$$
When no exterior forces act on the system \eqref{eqns}, the initial value problem reads

\begin{equation}\label{non force eqns}
\left\{\begin{split} \partial_t \vec u+\text{ div } (\vec u\otimes \vec u)&= \Delta\vec u+\vec J\wedge \vec B -\vec\nabla p\\  \partial_t\vec E-\vec\nabla\wedge\vec B&=-\vec J 
\\ \partial_t\vec B+\vec\nabla\wedge \vec E&=0
\\ \text{ div }\vec u= \text{ div }\vec  B&=0\\ \vec J&=\vec E+\vec u\wedge\vec B
\\ \vec u(0,.)=\vec u_0,\quad  \vec E(0,.)=\vec E_0,&\quad\vec B(0,.)=\vec B_0.
\end{split}\right.\end{equation}
Solutions to \eqref{non force eqns} formally enjoy the energy balance
\begin{eqnarray*} 
\label{energy}
\frac12  \frac{d}{dt} \big[\|\vec u\|_{L^2}^2 +\|\vec B\|_{L^2}^2 +\|\vec E\|_{L^2}^2
   \big]   + \|\vec J\|_{L^2}^2  + \|\nabla\vec u \|_{L^2}^2 = 0, 
\end{eqnarray*}
which suggests that weak solutions would exist in the space 
$$
\vec u\in L^\infty(0,\infty; L^2)\cap L^2(0,\infty; \dot H^1);\quad 
\vec E,\; \vec B\in L^\infty(0,\infty; L^2)\quad 
\vec J\in L^2(0,\infty; L^2).
$$
Unfortunately, weak solutions are not known to exist even in two space dimension. In \cite{Masmoudi10jmpa}, Masmoudi proved the existence of  global strong solutions. Later on,  Ibrahim and Keraani \cite{IK} relaxed the regularity condition on the initial data to construct global small solutions \`a la Kato. This result was recently improved by Germain, Ibrahim and Masmoudi \cite{GIM} by taking small initial data $\vec u_0$, $\vec E_0$ and $\vec B_0$ in $\dot H^{1/2}$, and construct a solution $(\vec u,\vec E,\vec B)$ such that $\vec u\in L^\infty\dot H^{1/2}\cap L^2\dot H^{3/2}\cap L^2L^\infty$, $\vec E\in L^\infty\dot H^{1/2}\cap L^2\dot H^{1/2}$ and $\vec B\in L^\infty\dot H^{1/2}$. \\

This paper is organized as follows. In section one, we introduce some useful notation and state our results: A first Theorem about the existence of time-periodic solutions in spaces of Sobolev type with an extra spatial regularity. Then, we relax the hypothesis of the first Theorem and extend it to critical Besov type spaces. A such an extension seems to us necessary in order to prove the last result about the stability of the periodic-in time solutions. Section two is devoted to the proof of the two existence results, and we only give the full details in the case of Sobolev. In section three, we start by examining a maximal regularity result adapted to the spaces that incorporates averaged decay  in time. Then, we show the decay of the electromagnetic field where we use the full spectral properties of the weakly damped Maxwell's equations. Then, we list all the nonlinear estimates that appear in the study of the nonlinear stability, and we only prove the worst two of them when two factors have no decay in time. The manuscript is then finished with an Appendix summarizing the main spectral properties of the weakly damped Maxwell's equations. 

\begin{center}
{\large Acknowledgements}
\end{center}
Slim Ibrahim is partially supported by NSERC Discovery grant \# 371637-2014. Nader Masmoudi is in part supported by NSF grant DMS-1211806\#\\
Slim Ibrahim would like to thank the University of \'Evry and the University of Paris Diderot-Paris 7 for their hospitality to accomplish a part of this work.



\section{Notation}
The well-known Littlewood-Paley decomposition and the corresponding frequency cut-off operators will be of frequent use in this paper. We briefly recall it to define the functional spaces we need.\\
There exists a radial positive  function   $\varphi\in\mathcal{D}(\R^d\backslash{\{0\}})$ such that
\begin{eqnarray*}
\label{lpfond1}
\sum_{q\in\mathbb Z} \varphi (2^{-q}\xi)& =& 1\qquad \forall\, \xi\in\mathbb R^d\setminus\{0\},
\\
\Supp\ \varphi(2^{-q}\cdot)\cap \Supp\ \varphi(2^{-j}\cdot)&=&\emptyset, \qquad \forall\, |q-j|\geq 2.
\end{eqnarray*}
For every  $q\in\mathbb Z$ and $v\in{\mathcal S}'(\R^d)$ we set 
$$
\Delta_qv=\mathcal{F}^{-1} \left[ \varphi(2^{-q}\xi) \hat v (\xi) \right] \quad\hbox{ and  }\;
S_q=\sum_{ j=-\infty}^{ q-1}\Delta_{j}.
$$
Bony's decomposition \cite{Bony81} consists in splitting the product  $uv$ into three parts\footnote{ It should be said that this decomposition is true in the class of distributions for which 
$\sum_{q\in\mathbb Z}\Delta_q=I$.
For example, polynomial functions do not belong to this class.}: 
$$
uv=T_u v+T_v u+R(u,v),
$$
with
$$
T_u v=\sum_{q}S_{q-1}u\Delta_q v,\quad  R(u,v)=\sum_{q}\Delta_qu\tilde\Delta_{q}v  \quad\hbox{and}\quad \tilde\Delta_{q}=\sum_{i=-1}^1\Delta_{q+i}.
$$
For $(p,r)\in[1,+\infty]^2$ and $s\in\R$ we define  the homogeneous Besov \mbox{space $\Dot B_{p,r}^s$} 
as the set of  $u\in\mathcal{S}'(\R^d)$ such that $u = \sum_q \Delta_q u$ and
$$
\|u\|_{\dot B^s_{p,r}} =\Bigl\|\left( 2^{qs} \|\Delta
_qu\|_{L^p}\right)_{q\in \Z}\Bigr\|_{\ell^r(\Z)}<\infty.
$$
In the case $p=r=2$, the space $ \dot B^s_{2,2}$ turns out to be the classical homogeneous Sobolev space $\dot H^s$.\\
In order to prove the our stability result, we need to build spaces that take into account  the different decay rates coming from the coupling of the two types of PDEs (Navier-Stokes, and Maxwell), and also the weak decay of the low frequencies in the Maxwell's equations. In addition, and in order to estimate the nonlinear terms, we need to introduce spaces that capture an average decay in time, and not just pointwise decay. This will be crucial in our analysis. We define the spaces that distinguish between the high and low frequencies of a function 
\begin{definition}
Let $\Delta_q$ denote the dyadic 
frequency localization operator defined in section 1. We define a space that distinguishes between the high and low frequencies of a function as follows. 
For $s_1, s_2\in\R$ and $1\leq p, q_1,q_2\leq \infty$ define the space ${\dot{\mathcal B}}^{s_1,s_2}_{p,q}$ by its norm
$$
 \|\phi\|_{\dot{\mathcal B}^{(s_1,s_2)}_{p,(q_1,q_2)}  }:= 
\big(\sum_{k\leq 0}2^{kqs_1}\|\Delta_k\phi\|_{L^p}^{q_1}\big)^\frac1{q_1}+ 
\big(\sum_{k>0}2^{kqs_2}\|\Delta_k\phi\|_{L^p}^{q_2}\big)^\frac1{q_2}.
 $$
We will also use the short-hands
$$
\mbox{and} \quad \dot B^s_{p,(q_1,q_2)} := 
\dot{\mathcal B}^{(s,s)}_{p,(q_1,q_2)},\quad \dot{H}^s = \dot{\mathcal B}^{(s,s)}_{2,(2,2)}, \quad 
\mbox{and} \quad \dot H^{s,t} := 
\dot{\mathcal B}^{(s,t)}_{2,(2,2)}.
$$ 
Finally, define the space-time functional space $\tilde L^r_T{   {\dot{\mathcal B}}^{(s_1,s_2)}_{p,(q_1,q_2)}}$ by its norm
$$
 \|\phi\|_{  \tilde L^r_T{   {\dot{\mathcal B}}^{(s_1,s_2)}_{p,(q_1,q_2)}   }}:= 
\big(\sum_{k\leq 0}2^{q_1s_1k}\|\Delta_k\phi\|_{L^r_TL^p}^{q_1}\big)^\frac1{q_1}+ 
\big(\sum_{k>0}2^{q_2ks_2}\|\Delta_k\phi\|_{L^r_TL^p}^{q_2}\big)^\frac1{q_2},
$$
with the trivial extension when $r=\infty$. We also define the new spaces that take into account an averaged decay in time. Precisely,  we denote 
\begin{eqnarray*}
\tilde{\sup_{n\in\mathbb N}}\;(n+1)^{\frac{1-\varepsilon}2}\|u\|_{  L^2{(n,n+1);   {\dot{\mathcal B}}^{(s_1,s_2)}_{p,(q_1,q_2)}   }}:&=&
\big(\sum_{k\leq 0}2^{q_1s_1k}\sup_{n\in\mathbb N}(n+1)^{\frac{(1-\varepsilon)q_1}2}\|\Delta_ku\|_{L^2(n,n+1;L^p)}^{q_1}\big)^\frac1{q_1}\\
&+& 
\big(\sum_{k\leq 0}2^{q_2s_2k}\sup_{n\in\mathbb N}(n+1)^{\frac{(1-\varepsilon)q_2}2}\|\Delta_ku\|_{L^2(n,n+1;L^p)}^{q_2}\big)
^\frac1{q_2}.
\end{eqnarray*}
with the obvious generalizations in the cases $q_j=\infty$, or $\tilde L^r_T  {\dot H}^{s}$ etc..
\end{definition}

The space ${\dot{\mathcal B}}^{(s_1,s_2)}_{p,(q_1,q_2)}$ is nothing but the usual 
Besov space ${\dot{\mathcal B}}^{s_2}_{p,q_2}$ for high frequencies while it behaves like 
${\dot{\mathcal B}}^{s_1}_{p,q_1}$ for low frequencies. If $s_1>s_2$, it is not difficult to 
see that ${\dot{\mathcal B}}^{(s_1,s_2)}_{p,(q_1,q_2)}={\dot{\mathcal B}}^{s_1}_{p,q_1}+
{\dot{\mathcal B}}^{s_2}_{p,q_2}$. 
The $\tilde L$ type spaces were first used by Chemin and Lerner \cite{CL95}.

In the sequel, consider a parameter $0<\varepsilon<1$, introduce the spaces $\mathcal X_1$, $\mathcal X_2$, $\mathcal X_3$, 
and their ``dual'' counterparts $\mathcal Y_1$ and $\mathcal Y_2$ by defining their norms.
$$
\|u\|_{\mathcal X_1}:= \tilde{\sup_{t>0}}\;(t+1)^\frac{1-\varepsilon}2
\|u(t)\|_{  {\dot{\mathcal B}}^{(\frac32-\varepsilon,\frac12)}_{2,(\infty,1)}}+
\tilde{\sup_{n\in\mathbb N}}\; (n+1)^{\frac{1-\varepsilon}2}
\|u\|_{L^2(n,n+1;{\dot{\mathcal B}}^{\frac32}_{2,(\infty,1)})}
$$

$$
\|E\|_{\mathcal X_2}:= \tilde{\sup_{t>0}}\; (t+1)^{\frac{1-\varepsilon}2}\|E(t)\|_{  H^\frac12}\sim 
\tilde{\sup_{n\in\mathbb N}}\;(n+1)^{\frac{1-\varepsilon}2} \|E\|_{L^\infty(n,n+1;H^\frac12)}
$$

$$
\|B\|_{\mathcal X_3}:= \tilde{\sup_{t>0}}\; (t+1)^{\frac{1-\varepsilon}2}\|B(t)\|_{ \dot H^{1,\frac12}}\sim 
\tilde{\sup_{n\in\mathbb N}}\;(n+1)^{\frac{1-\varepsilon}2} \|B\|_{L^\infty(n,n+1;\dot H^{1,\frac12})},
$$

and 
$$
\|F\|_{\mathcal Y_1}:=\tilde{\sup_{n\in\mathbb N}}\; (n+1)^{\frac{1-\varepsilon}2}
\|F\|_{L^2(n,n+1;{\dot{\mathcal B}}^{(-\frac12-\varepsilon,-\frac12)}_{2,(\infty,1)})},
$$

$$
\|G\|_{\mathcal Y_2}:=\tilde{\sup_{n\in\mathbb N}}\; (n+1)^{\frac{1-\varepsilon}2}\|G\|_{L^2(n,n+1;H^\frac12)}.
$$
Finally, let $\mathcal X:=(\mathcal X_1\cap\tilde L^\infty(\dot{\mathcal B}^\frac12_{2,(\infty,1))})\times\mathcal X_2\times\mathcal X_3$.

\subsection{Results}
In our first result and under a smallness assumption on the forces, we construct (in Sobolev spaces with an extra $\delta$ regularity) a time-periodic solution $\vec\Gamma_{\rm per}$ to \eqref{eqns}. More precisely, we have

\begin{theorem}\label{theoNSM}$\ $\\
Let $0<\delta<1$. Then there exists a positive constant $\epsilon_{T,\delta}$ such that~:  if the time-periodic forces $\vec F_{\rm per}$, $\vec G_{\rm per}$ and $\vec H_{\rm per}$ satisfy the following assumptions :
\begin{enumerate}
\item $\vec F_{\rm per}$ belongs to $L^2_{\rm per} \dot H^{-\frac{1}{2}} \cap L^2_{\rm per} \dot  H^{-\frac{1}{2}+\delta}$ and
$$  \sqrt{\int_0^T \|\vec F_{\rm per}(t,.)\|_{  \dot H^{-\frac{1}{2}}}^2\, dt}+ \sqrt{\int_0^T \|\vec F_{\rm per}(t,.)\|_{  \dot H^{-\frac{1}{2}+\delta}}^2\, dt}<\epsilon_{T,\delta}$$
\item  the mean value  $\vec F_0=\frac{1}{T}\int_0^T \vec F_{\rm per} (t,.)\, dt$   belongs  to $ \dot B^{-\frac{3}{2}}_{2,\infty}$ and 
$$\|\vec F_0\|_{\dot B^{-3/2}_{2,\infty}}<\epsilon_{T,\delta}$$
\item $\vec G_{\rm per}$ belongs to $  L^2_{\rm per} H^{\frac{1}{2}+\delta}$ and
$$  \sqrt{\int_0^T \|\vec G_{\rm per}(t,.)\|_{    H^{\frac{1}{2}+\delta}}^2\, dt}<\epsilon_{T,\delta}$$
\item the mean value  $\vec G_0=\frac{1}{T}\int_0^T \vec G_{\rm per}(t,.)\, dt$   belongs to $\dot H^{-1}$ and
$$ \|\vec G_0\|_{\dot H^{-1} }<\epsilon_{T,\delta},$$
\item $\vec H_{\rm per}$ is divergence-free ($\text{\rm div }\vec H_{\rm per}=0$)  and  $\vec H_{\rm per}$ belongs to $  L^2_{\rm per} H^{\frac{1}{2}+\delta}$ with
$$  \sqrt{\int_0^T \|\vec H_{\rm per}(t,.)\|_{    H^{\frac{1}{2}+\delta}}^2\, dt}<\epsilon_{T,\delta}$$
\item the mean value    $\vec H_0=\frac{1}{T}\int_0^T \vec H_{\rm per}(t,.)\, dt$ belongs to $\dot H^{-2}$ and
$$  \|\vec H_0\|_{\dot H^{-2}}<\epsilon_{T,\delta},$$\end{enumerate}
 then the Navier--Stokes--Maxwell problem (\ref{eqns}) has a time-periodic solution $(\vec u_{\rm per}, \vec E_{\rm per},\vec B_{\rm per})$ such that :
 \begin{itemize}
 \item $\vec u_{\rm per}$ belongs to $L^\infty_{\rm per} \dot B^{\frac{1}{2}}_{2,\infty}\cap L^2_{\rm per} \dot  H^{\frac{3}{2}+\delta}$
 \item $\vec E_{\rm per}$ and $\vec B_{\rm per}$ belong  to $  L^\infty_{\rm per} H^{\frac{1}{2}+\delta}$.
 \end{itemize}
\end{theorem}
We extend the above statement to solutions in critical spaces of Besov-type. This will be crucial for the stability as we were not able to prove the stability in the spaces given by Theorem \ref{theoNSM}. More precisely, we have

\begin{theorem}
\label{MAIN1} 

$\;$

Let $T>0$ denote the time period of three periodic forces 
$\vec F_{\text{per}}$, $\vec G_{\text{per}}$ and $\vec H_{\text{per}}$ decomposed as follows into a fluctuating and zero mean parts:

$\vec F_{\text{per}}(t,x):=\vec F_0(x)+\vec F_f(t,x)$, $\vec G_{\text{per}}:=\vec G_0(x)+\vec G_f(t,x)$ and 
$\vec H_{\text{per}}:=\vec H_0(x)+\vec H_f(t,x)$ with
$$
\int_0^T\vec F_f\;dt=\int_0^T\vec G_f\;dt=\int_0^T\vec H_f\;dt=0.
$$
There exists  $\varepsilon_T>0$ such that under the following smallness assumptions
$$
\|\vec F_{\text{per}}\|_{\tilde L^2(0,T;{\dot{\mathcal B}}^{-\frac12}_{2,(\infty,1)})}+
\|F_0\|_{{\dot{\mathcal B}}^{-\frac32}_{2,(\infty,1)}}\leq\varepsilon_T
$$

$$
\|\vec G_{\text{per}}\|_{L^2(0,T;H^{\frac12})}+\|\vec G_0\|_{\dot H^{-1}}\leq\varepsilon_T
$$

and 
$$
\|\vec H_{\text{per}}\|_{L^2(0,T;H^{\frac12})}+\|\vec H_0\|_{\dot H^{-2}}\leq\varepsilon_T,
$$
a unique mild solution $\vec\Gamma_{\text{per}}=(\vec u_{\text{per}},\vec E_{\text{per}},\vec B_{\text{per}})$ 
of \eqref{eqns} exists such that $\vec u_{\text{per}}\in \tilde L^\infty_{\text{per}}\dot{\mathcal B}^\frac12_{2,(\infty,1)}\cap \tilde L^2_{\text{per}}\dot{\mathcal B}^\frac32_{2,(\infty,1)}$ 
and  $\vec E_{\text{per}}, \vec B_{\text{per}}\in\tilde L^\infty_{\text{per}}{ H}^\frac12$.
\end{theorem}
Another variant of the existence result of time periodic solutions is given by the following theorem where we require a slightly better control of the high frequencies of the solution $\vec\Gamma$.

\begin{rem}
\label{rem1}
\begin{itemize}
\item
One can prove local existence if the low frequency part of the initial data of the velocity field is in 
$\dot{\mathcal B}^\frac12_{2,(\infty,1)}$  and its high frequency is in $\dot H^\frac12$.
\item Compared to the results of  \cite{GIM} in the absence of forcing terms, the statement of Theorem \ref{theoNSM} requires a slightly better control of high frequencies for $\vec u$, $\vec E$ and $\vec B$, a better control of low frequencies of $\vec E$ and $\vec B$, and a weaker control on the low frequencies of $\vec u$.
\item Our proof also shows that the more regular is the forcing, the more regular will be its corresponding  periodic-in time solution. Indeed, if for example the forcing is small in 
$$
\vec F_{\rm per}\in L^2_{\rm per} \dot H^{-\frac{1}{2}}\cap L^2_{\rm per} \dot  H^{\frac{1}{2}}, \quad \vec G_{\rm per}\in L^2_{\rm per} H^{\frac{3}{2}},\quad
 \vec H_{\rm per}\in  L^2_{\rm per} H^{\frac{3}{2}},
 $$
then, we have\\
$\bullet$ $\vec u_{\rm per}$ belongs to $L^\infty_{\rm per} \dot B^{\frac{1}{2}}_{2,\infty} \cap L^2_{\rm per}( \dot  H^{\frac{3}{2}}\cap \dot  H^{\frac{5}{2}})$\\
 $\bullet$ $\vec E_{\rm per}$ and $\vec B_{\rm per}$ belong  to $  L^\infty_{\rm per} H^{\frac{3}{2}}$.
\end{itemize}
\end{rem}
Next, we study the stability of the time-periodic solutions: what happens when, at some time $t_0$, one takes a perturbation of the solutions constructed in above: 
$$
\vec \Gamma(t_0)=\vec \Gamma_{\rm per}(t_0)+\vec \Gamma_{\rm err},
$$ 
with $\vec \Gamma_{\rm err}$ small in $\dot B^{1/2}_{2,(\infty,1)}\times H^{1/2}\times H^{1/2}$?
Do we have a global solution of \eqref{eqns} on $[t_0,+\infty)$ and does the error go to $0$ in suitable norms when $t$ goes to $+\infty$? 
The main problem rises when we estimate the cross terms coming from the interactions between the periodic solution and the solution we want to construct. The worst interaction is given by a 
term of the type 
\begin{eqnarray}
\label{bad term}
(\vec U_{\rm per}\wedge \vec B_{\rm per})\wedge \vec B,
\end{eqnarray}
first because of the non-decay of $U_{\rm per}$, and $B_{\rm per}$, and second because we barley miss an $L^\infty(L^\infty)$ estimate on $U_{\rm per}$. To overcome such a problem, we  impose a strong condition on the low frequencies of the velocity field. In doing so, we are obliged to allow an-$\varepsilon$ loss in the time decay rate. It is important to notice that because of this problem, we were not able to show the stability of the  the solutions given by Theorem \ref{theoNSM}. Hence, our  second main result is the following.

\begin{theorem}
\label{MAIN2} 
Given three $T$-periodic forces $\vec F_{\text{per}}(t)$, $\vec G_{\text{per}}(t)$ and $\vec H_{\text{per}}(t)$ 
satisfying the hypothesis of Theorem \ref{MAIN1}, and denote by $\vec \Gamma_{\text{per}}(t)$ the corresponding small $T$-periodic solution.
Consider an initial data $\vec \Gamma^0_{\text{err}}+\vec\Gamma_{\text{per}}(0)$ with $\vec \Gamma^0_{\text{err}}$ small in 
$\dot{\mathcal B}^\frac12_{2,(\infty,1)}\times H^\frac12\times H^\frac12$, there exists $\vec{\bar\Gamma}$ 
a global solution of \eqref{eqns} with that initial data $\vec\Gamma^0_{\text{err}}+\vec\Gamma_{\text{per}}(0)$. Moreover, we have 
$$
\vec{\bar\Gamma}-\vec\Gamma_{\text{per}}\in \mathcal X,
$$
so that $\vec{\bar\Gamma}$ converges asymptotically to $\vec\Gamma_{\text{per}}$ as $t$ goes to infinity.
\end{theorem}

Our proof of Theorem \ref{theoNSM} relies on several linear and nonlinear estimates (product rules in Besov and Sobolev spaces) and uses Fourier series expansion of the time-periodic solution. Such an expansion was used in \cite{Kyed} for the time-periodic forced Navier-Stokes. The proof of Theorem \ref{MAIN2} then goes through a fixed point argument in a suitable space. In order to have the asymptotic convergence, the functional space has to include decay properties. Thus, we are required to exhibit the decay from the velocity and the electro-magnetic fields. Both the dissipation coming from the viscosity of the fluid and from the resistivity in Ohm's law, enable us to have some decay for the velocity $\vec u$ and the electric field $\vec E$. To qualitatively transfer such a decay to the magnetic field is not as easy and clear as for $\vec u$ and $\vec E$. In \cite{IK}, and then \cite{GIM}, a weak decay of the magnetic field was proven in both space dimension two and three. However, the decay was not used (in three space dimension) to construct global small solution. In the contrary, here the use of the decay is an essential fact to show asymptotic convergence.

 \section{Construction of time-periodic solutions}
The purpose of this section is to prove Theorem \ref{MAIN1} and Theorem \ref{theoNSM}. Since the method is the same for both but details are much more involved in critical spaces (the Besov case), and for the sake of simplicity, we opted to give the full details of the proof of Theorem \ref{theoNSM} and only sketch the necessary changes to complete the proof of Theorem  \ref{MAIN1}. First, we introduce some useful notation

$\ $

\noindent{\bf Notation : } \\
For $0<\delta<1$, we shall write 
\begin{itemize}
\item $(\vec u,\vec E,\vec B)\in \mathbb{X} $ if   \begin{enumerate}
 \item $\vec u $ belongs to $L^\infty_{\rm per} \dot B^{\frac{1}{2}}_{2,\infty}\cap L^2_{\rm per} \dot  H^{\frac{3}{2}+\delta}$
 \item $\vec E $ and $\vec B $ belong  to $  L^\infty_{\rm per} H^{\frac{1}{2}+\delta}$.
 \end{enumerate}
\item $(\vec F,\vec G,\vec H)\in \mathbb{Y}$ if 
\begin{enumerate}
\item $\vec F $ belongs to $L^2_{\rm per} \dot H^{-\frac{1}{2}} \cap L^2_{\rm per} \dot  H^{-\frac{1}{2}+\delta}$
\item  the mean value  $\vec F_0=\frac{1}{T}\int_0^T \vec F (t,.)\, dt$   belongs  to $ \dot B^{-\frac{3}{2}}_{2,\infty}$
 \item $\vec G $ belongs to $  L^2_{\rm per} H^{\frac{1}{2}+\delta}$ 
 \item the mean value  $\vec G_0=\frac{1}{T}\int_0^T \vec G (t,.)\, dt$   belongs to $\dot H^{-1}$
\item $\vec H $ is divergence-free ($\text{\rm div }\vec H =0$)  and  $\vec H $ belongs to $  L^2_{\rm per} H^{\frac{1}{2}+\delta}$  
\item the mean value   $\vec H_0=\frac{1}{T}\int_0^T \vec H (t,.)\, dt$ belongs to $\dot H^{-2}$. \end{enumerate}
\end{itemize}
\subsection{Proof of theorem \ref{theoNSM}.}

The problem is solved by a Picard iterative scheme:  $(\vec u_{\rm per},\vec E_{\rm per},\vec B_{\rm per})$ will be the limit of the time-periodic functions $(\vec U_n,\vec E_n,\vec B_n)$ solving the system
\begin{equation}\label{eqpic}
\left\{\begin{split} \partial_t \vec U_{n+1}-  \Delta\vec U_{n+1}&= \vec F_n \\  \partial_t\vec E_{n+1}-\vec\nabla\wedge\vec B_{n+1}+\vec E_{n+1}&=\vec G_n
\\ \partial_t\vec B_{n+1}+\vec\nabla\wedge \vec E_{n+1}&=\vec H_n
\\ \text{ div }\vec u_{n+1}= \text{\rm div }\vec  B_{n+1}&=0 
\end{split}\right.\end{equation}
with
\begin{equation} 
\left\{\begin{split} \vec U_0=0,\quad \vec E_0=0,&\quad \vec B_0=0\\  \vec F_n&=\mathbb{P}\left(-\text{ div } (\vec U_n\otimes \vec U_n)+\vec E_n\wedge \vec B_n+(\vec U_n\wedge\vec B_n)\wedge\vec B_n+\vec F_{\rm per}\right)\\  \vec G_n&=-\vec U_n\wedge\vec B_n+\vec G_{\rm per}
\\ \vec H_n&=\vec H_{\rm per}
\end{split}\right.\end{equation}
where $\mathbb{P}$ is the Leray projection operator on solenoidal vector fields.

The first Lemma gives product rules in Sobolev spaces that close the iterative scheme. More precisely, we have.
\begin{lemma}
\label{law prod lem}
 If $(\vec u,\vec E,\vec B)\in \mathbb{X} $ and $(\vec F,\vec G,\vec H)\in \mathbb{Y}$,  define  $(\vec F_1,\vec G_1,\vec H_1) $ as
 \begin{equation} 
\left\{\begin{split}   \vec F_1&=\mathbb{P}\left(-\text{\rm div } (\vec u\otimes \vec u)+\vec E\wedge \vec B+(\vec u\wedge\vec B)\wedge\vec B+\vec F\right)\\  \vec G_1&=-\vec u\wedge\vec B +\vec G 
\\ \vec H_1&=\vec H .
\end{split}\right.\end{equation}
Then we have $(\vec F_1,\vec G_1,\vec H_1)\in \mathbb{Y}$.
 \end{lemma} 
 
 \begin{proof}  Point-wise product maps $\dot B^{3/2}_{2,1}\times H^{\frac{1}{2}+\delta}$ to $H^{\frac{1}{2}+\delta}$. As 
 $L^\infty_{\rm per} \dot B^{\frac{1}{2}}_{2,\infty}\cap L^2_{\rm per} \dot  H^{\frac{3}{2}+\delta} \subset L^2_{\rm per}\dot B^{3/2}_{2,1}$, we find that $\vec G_1 $ belongs to $  L^2_{\rm per} H^{\frac{1}{2}+\delta}$. 
 
 Moreover, pointwise product maps $\dot B^{1/2}_{2,\infty}\times L^2$ to $\dot H^{-1}$.  Thus, we find that $\vec u\wedge\vec B$ belongs to $L^\infty_{\rm per}\dot H^{-1}$ and the mean value of $\vec G_1 $ belongs to $ \dot H^{-1}$.  
 
 For $s> 0$, we have $$\|fg\|_{\dot H^s}\leq C_s (\|f\|_{\dot B^{1/2}_{2,\infty}} \|g\|_{\dot H^{s+1}}+\|g\|_{\dot B^{1/2}_{2,\infty}} \|f\|_{\dot H^{s+1}}.$$ Using the embedding  $L^\infty_{\rm per} \dot B^{\frac{1}{2}}_{2,\infty}\cap L^2_{\rm per} \dot  H^{\frac{3}{2}+\delta} \subset L^2_{\rm per}\dot  H^{3/2} $, we find that  the term $\text{\rm div }(\vec u\otimes\vec u)$ belongs to $L^2_{\rm per} \dot H^{-\frac{1}{2}} \cap L^2_{\rm per} \dot  H^{-\frac{1}{2}+\delta}$. Moreover, pointwise product maps $\dot B^{1/2}_{2,\infty}\times \dot B^{1/2}_{2,\infty}$ to $\dot B^{-1/2}_{2,\infty}$, so that  $\text{\rm div }(\vec u\otimes\vec u)$ belongs to $L^\infty_{\rm per} \dot B^{-3/2}_{2,\infty}$ and the mean value of  $ \text{\rm div }(\vec u\otimes\vec u)$ belongs to $ \dot B^{-3/2}_{2,\infty}$. 
 
 We have $\vec E+\vec u\wedge\vec B\in  L^2_{\rm per}  H^{\frac{1}{2}+\delta} \subset  L^2_{\rm per}\dot H^{\frac{1}{2}}\cap L^2_{\rm per}\dot H^{\frac{1}{2}+\delta}$ and $B\in L^\infty_{\rm per} H^{\frac{1}{2}+\delta}\subset L^\infty_{\rm per} \dot H^{1/2}$. This gives $(\vec E+\vec u\wedge\vec B)\wedge\vec B \in L^2_{\rm per} \dot H^{-\frac{1}{2}} \cap L^2_{\rm per} \dot  H^{-\frac{1}{2}+\delta}$.   Moreover, we have $\vec E+\vec u\wedge\vec B\in  L^2_{\rm per}  H^{\frac{1}{2}+\delta} \subset  L^2_{\rm per}L^2$ and  $B\in L^\infty_{\rm per} H^{\frac{1}{2}+\delta}\subset L^\infty_{\rm per}L^2$, so that  $(\vec E+\vec u\wedge\vec B)\wedge\vec B \in L^2_{\rm per}  L^1 \subset L^2_{\rm per} \dot  B^{-3/2}_{2,\infty}$ and we find that the mean value of  $(\vec E+\vec u\wedge\vec B)\wedge\vec B $ belongs to $ \dot  B^{-3/2}_{2,\infty}$.  
 Thus, the lemma is proved.  
 \end{proof}
 The second Lemma shows that given a time-periodic forcing, we can construct, in the right functional space, a solution to the linear problem in the iterative scheme.  Indeed,
 
 \begin{lemma}\label{lemXY}$\ $\\
 Let  $(\vec F,\vec G,\vec H)\in \mathbb{Y}$.  Then the time-periodic solution $\vec\Gamma:=(\vec u,\vec E,\vec B)$ of the system
 \begin{equation} \label{syst}
\left\{\begin{split} \partial_t \vec  u-  \Delta\vec u&= \vec F \\  \partial_t\vec E-\vec\nabla\wedge\vec B +\vec  E &=\vec G
\\ \partial_t\vec B +\vec\nabla\wedge \vec E &=\vec H 
\\   \text{\rm  div }\vec  B &=0 
\end{split}\right.\end{equation}
 satisfies $(\vec u,\vec E,\vec B)\in\mathbb{X}$.
 \end{lemma}
 
 \begin{proof} We follow the formalism of Kyed \cite{Kyed} and expand $\vec F$, $\vec G$, $\vec H$, $\vec u$, $\vec E$ and $\vec B$ as time Fourier series:
  \begin{equation} 
\left\{\begin{split}  \vec F&=  \sum_{k\in\mathbb{Z}} \vec F_k(x) e^{ik\frac{2\pi}{T}t}
\\  \vec G&=  \sum_{k\in\mathbb{Z}} \vec G_k(x) e^{ik\frac{2\pi}{T}t}
\\  \vec H&=  \sum_{k\in\mathbb{Z}} \vec H_k(x) e^{ik\frac{2\pi}{T}t}
\\  \vec u&=  \sum_{k\in\mathbb{Z}} \vec U_k(x) e^{ik\frac{2\pi}{T}t}
\\  \vec E&=  \sum_{k\in\mathbb{Z}} \vec E_k(x) e^{ik\frac{2\pi}{T}t}
\\  \vec B&=  \sum_{k\in\mathbb{Z}} \vec B_k(x) e^{ik\frac{2\pi}{T}t}.
\end{split}\right.\end{equation}
Define $\hat F_k, \hat G_k, \hat H_k, \hat U_k, \hat E_k, \hat B_k$ the Fourier transforms, in space, of $\vec F_k, \vec G_k, \vec H_k, \vec U_k, \vec E_k, \vec B_k$.
First, we explicitly solve for $(\vec U,\vec E,\vec B)$. From $\partial_t \vec  u-  \Delta\vec u = \vec F $, we get :
   \begin{equation}  \hat U_k(\xi)= \frac{1}{\vert\xi\vert^2+i\frac{2\pi}{T}k}\hat F_k(\xi).\end{equation}
   From $$\partial_t^2\vec B=\partial_t\vec H-\vec\nabla\wedge \partial_t\vec E=\Delta\vec  B-\vec\nabla\wedge\vec G+\vec H-\partial_t \vec B,$$ we get
      \begin{equation}  \hat B_k(\xi)= \frac{1}{\vert\xi\vert^2-\frac{4\pi^2}{T^2}k^2+i\frac{2\pi}{T}k} \left( (1+i\frac{2\pi}{T}k) \hat H_k(\xi)-i\vec\xi\wedge\hat G_k(\xi) \right).\end{equation}
      From $\partial_t\vec E+\vec E=\vec G+\vec\nabla\wedge \vec B$, we get
      $$ \hat E_k(\xi)=\frac{1}{1+i\frac{2\pi}{T}k} (\hat G_k(\xi)+i\vec\xi\wedge \hat B_k(\xi)).$$
      If we decompose $\vec E_k$ into its solenoidal part $\vec E_{k,\sigma}$ and its irrotational part $\vec E_{k,\nabla}=\vec\nabla\frac{1}{\Delta}\text{\rm div }\vec E_k$, and similarly write $\vec G_k=\vec G_{k,\sigma}+\vec G_{k,\nabla}$, we get
            \begin{equation}  \hat E_{k,\sigma}(\xi)= \frac{1}{\vert\xi\vert^2-\frac{4\pi^2}{T^2}k^2+i\frac{2\pi}{T}k} \left(  i\vec\xi\wedge \hat H_k(\xi) + i\frac{2\pi}{T} k\, \hat G_{k,\sigma}(\xi) \right)\end{equation}
            and
                  \begin{equation}  \hat  E_{k,\nabla}(\xi)= \frac{1}{1+i\frac{2\pi}{T}k} \hat G_{k,\nabla}(\xi) .\end{equation}
      
Then, we proceed to estimate the solution. We are going to separately estimate  the time averages $\vec U_0$, $\vec E_0$ and $\vec B_0$ and the fluctuation components $\vec U_f=\vec u-\vec U_0$, $\vec E_f=\vec E-\vec E_0$ and $\vec B_f=\vec B-\vec B_0$. \\
Notice that $\vec U_0=-\frac{1}{\Delta}\vec F_0$, and since $\vec F_0\in \dot B^{-3/2}_{2,\infty}\cap \dot H^{-\frac{1}{2}+\delta}$, we get $\vec U_0\in \dot B^{1/2}_{2,\infty}\cap \dot  H^{\frac{3}{2}+\delta}$.\\       
Similarly, we have $\vec B_0=\frac{1}{\Delta}\vec\nabla\wedge \vec G_0-\frac{1}{\Delta} \vec H_0$.  Since $\vec G_0\in \dot H^{-1}\cap H^{\frac{1}{2}+\delta}$ and $\vec H_0\in \dot H^{-2}\cap  H^{\frac{1}{2}+\delta}$, we find that $\vec B_0\in H^{\frac{1}{2}+\delta}$. \\
We have $\vec E_0=\vec G_{0,\nabla}-\frac{1}{\Delta}\vec\nabla\wedge \vec H_0$. Since $\vec G_0\in H^{\frac{1}{2}+\delta}$ and $\vec H_0\in \dot H^{-2}\cap H^{\frac{1}{2}+\delta}$, we find that $\vec E_0\in H^{\frac{1}{2}+\delta}$.\\
Next, we estimate $\vec U_f$. We have
     \begin{equation*}\begin{split}
     \int_0^T \|\vec U_f\|_{\dot H^{\frac{3}{2}+\delta}}^2\; dt &= T\sum_{k\neq 0} \|\vec U_k\|_{\dot H^{\frac{3}{2}+\delta}}^2\\ &= \frac{T}{(2\pi)^3} \sum_{k\neq 0} \int_{\mathbb{R}^3} \frac{\vert\xi\vert^{3+2\delta}}{\vert\xi\vert^4+ \frac{4\pi^2}{T^2}k^2} \vert\hat F_k(\xi)\vert^2\, d\xi
     \\ &=  \frac{T}{(2\pi)^3} \sum_{k\neq 0} \int_{\mathbb{R}^3} \frac{\vert\xi\vert^{4}}{\vert\xi\vert^4+ \frac{4\pi^2}{T^2}k^2} \vert \xi\vert^{-1+2\delta} \vert\hat F_k(\xi)\vert^2\, d\xi\\ &\leq   \int_0^T \|\vec F_f\|_{\dot H^{-\frac{1}{2}+\delta}}^2\; dt
     \end{split}\end{equation*} so that $\vec U_f\in L^2_{\rm per} \dot H^{\frac{3}{2}+\delta}$. On the other hand, we have :
          \begin{equation*}\begin{split}
    \|\vec U_f(t,.)\|_{\dot H^{\frac{1}{2}}}^2&=   \frac{1}{(2\pi)^3}\int_{\mathbb{R}^3}  \vert\xi\vert\,  \left\vert \
     \sum_{k\neq 0} \hat  U_k(\xi)\, 
     e^{i\frac{2\pi}{T}kt}   
     \right\vert^2\, d\xi
     \\ &\leq  \frac{1}{(2\pi)^3}  \int_{\mathbb{R}^3} (\sum_{k\neq 0} \frac{\vert\xi\vert^{2}}{\vert\xi\vert^4+ \frac{4\pi^2}{T^2}k^2} ) (\sum_{k\neq 0}\frac{\vert\hat F_k(\xi)\vert^2}{\vert\xi\vert})\, d\xi
     \\ &\leq  A_T  \int_0^T \|\vec F_f\|_{\dot H^{-\frac{1}{2}}}^2\; dt
     \end{split}\end{equation*}
     with 
     $$ A_T=\frac{1}{T}\sup_{\xi\in\mathbb{R}^3} \sum_{k\neq 0} \frac{\vert\xi\vert^{2}}{\vert\xi\vert^4+ \frac{4\pi^2}{T^2}k^2} .
     $$
Similarly, we have
      \begin{equation*}\begin{split}
    \|\vec B_f(t,.)\|_{  H^{\frac{1}{2}+\delta}}^2&=   \frac{1}{(2\pi)^3}\int_{\mathbb{R}^3}  (1+\vert\xi\vert^2)^{\frac{1}{2}+\delta}\,  \left\vert \
     \sum_{k\neq 0} \hat  B_k(\xi)\, 
     e^{i\frac{2\pi}{T}kt}   
     \right\vert^2\, d\xi
     \\ &\leq  \frac{2}{(2\pi)^3}  \int_{\mathbb{R}^3}  (1+\vert\xi\vert^2)^{\frac{1}{2}+\delta}\,  (\sum_{k\neq 0} \frac{\vert\xi\vert^{2}}{\left(\vert\xi\vert^2- \frac{4\pi^2}{T^2}k^2\right)^2+ \frac{4\pi^2}{T^2}k^2} ) (\sum_{k\neq 0}\ \vert\hat G_k(\xi)\vert^2 )\, d\xi
      \\ &\ +  \frac{2}{(2\pi)^3}  \int_{\mathbb{R}^3} (1+\vert\xi\vert^2)^{\frac{1}{2}+\delta}\,  (\sum_{k\neq 0} \frac{1+ \frac{4\pi^2}{T^2}k^2}{\left(\vert\xi\vert^2- \frac{4\pi^2}{T^2}k^2\right)^2+ \frac{4\pi^2}{T^2}k^2} ) (\sum_{k\neq 0}\ \vert\hat H_k(\xi)\vert^2 )\, d\xi
     \\ &\leq  B_T (  \int_0^T \|\vec G_f\|_{  H^{\frac{1}{2}+\delta}}^2\; dt +  \int_0^T \|\vec H_f\|_{  H^{\frac{1}{2}+\delta}}^2\; dt)
     \end{split}\end{equation*}
with
$$ B_T=\sup_{\xi\in\mathbb{R}^3} \frac{1}{T} \sum_{k\neq 0}  \frac{1+ \vert\xi\vert^2+\frac{4\pi^2}{T^2}k^2}{\left(\vert\xi\vert^2- \frac{4\pi^2}{T^2}k^2\right)^2+ \frac{4\pi^2}{T^2}k^2}.
$$
We  have as well 
      \begin{equation*}\begin{split}
    \|\vec E_{\sigma,f}(t,.)\|_{  H^{\frac{1}{2}+\delta}}^2&=   \frac{1}{(2\pi)^3}\int_{\mathbb{R}^3}  (1+\vert\xi\vert^2)^{\frac{1}{2}+\delta}\,  \left\vert \
     \sum_{k\neq 0} \hat  E_{\sigma,k}(\xi)\, 
     e^{i\frac{2\pi}{T}kt}   
     \right\vert^2\, d\xi
     \\ &\leq  \frac{2}{(2\pi)^3}  \int_{\mathbb{R}^3}  (1+\vert\xi\vert^2)^{\frac{1}{2}+\delta}\,  (\sum_{k\neq 0} \frac{\vert\xi\vert^{2}}{\left(\vert\xi\vert^2- \frac{4\pi^2}{T^2}k^2\right)^2+ \frac{4\pi^2}{T^2}k^2} ) (\sum_{k\neq 0}\ \vert\hat H_k(\xi)\vert^2 )\, d\xi
      \\ &\ +  \frac{2}{(2\pi)^3}  \int_{\mathbb{R}^3} (1+\vert\xi\vert^2)^{\frac{1}{2}+\delta}\,  (\sum_{k\neq 0} \frac{ \frac{4\pi^2}{T^2}k^2}{\left(\vert\xi\vert^2- \frac{4\pi^2}{T^2}k^2\right)^2+ \frac{4\pi^2}{T^2}k^2} ) (\sum_{k\neq 0}\ \vert\hat G_{\sigma,k}(\xi)\vert^2 )\, d\xi
     \\ &\leq  C_T (  \int_0^T \|\vec G_{\sigma,f}\|_{  H^{\frac{1}{2}+\delta}}^2\; dt +  \int_0^T \|\vec H_f\|_{  H^{\frac{1}{2}+\delta}}^2\; dt)
     \end{split}\end{equation*}
with
$$ C_T=\sup_{\xi\in\mathbb{R}^3} \frac{1}{T} \sum_{k\neq 0}  \frac{ \vert\xi\vert^2+\frac{4\pi^2}{T^2}k^2}{\left(\vert\xi\vert^2- \frac{4\pi^2}{T^2}k^2\right)^2+ \frac{4\pi^2}{T^2}k^2}.
$$
Finally, we have
      \begin{equation*}\begin{split}
    \|\vec E_{\nabla,f}(t,.)\|_{  H^{\frac{1}{2}+\delta}}^2&=   \frac{1}{(2\pi)^3}\int_{\mathbb{R}^3}  (1+\vert\xi\vert^2)^{\frac{1}{2}+\delta}\,  \left\vert \
     \sum_{k\neq 0} \hat  E_{\nabla,k}(\xi)\, 
     e^{i\frac{2\pi}{T}kt}   
     \right\vert^2\, d\xi
     \\ &\leq  \frac{1}{(2\pi)^3}    \int_{\mathbb{R}^3} (1+\vert\xi\vert^2)^{\frac{1}{2}+\delta}\,  (\sum_{k\neq 0} \frac{ 1}{ 1+ \frac{4\pi^2}{T^2}k^2} ) (\sum_{k\neq 0}\ \vert\hat G_{\sigma,k}(\xi)\vert^2 )\, d\xi
     \\ &\leq  D_T  \int_0^T \|\vec G_{\nabla,f}\|_{  H^{\frac{1}{2}+\delta}}^2\; dt       \end{split}\end{equation*}
with
$$ D_T=\sup_{\xi\in\mathbb{R}^3} \frac{1}{T} \sum_{k\neq 0}  \frac{  1}{ 1+ \frac{4\pi^2}{T^2}k^2}.$$

$\ $
Thus, in order to finish the proof of the lemma, we need only to check  that  $A_T$, $B_T$, $C_T$ and $D_T$ are finite.  Equivalently, we must check that $$\alpha_0=\sup_{t\geq 0} \sum_{k=1}^{+\infty} \frac{t}{t^2+k^2}<+\infty$$ and
$$\beta_0=\sup_{t\geq 0}  \sum_{k=1}^{+\infty} \frac{t+k^2}{k^2+(t-k^2)^2} <+\infty.$$ If $t\leq 1/2$, we have $\frac{t}{t^2+k^2}\leq \frac{1}{2k^2}$ and $ \frac{t+k^2}{k^2+(t-k^2)^2} \leq \frac{3}{2k^2}$ and the control of the sum is easy. Thus, we consider only the case $t>1/2$.\\
We write
$$ \sum_{k=1}^{+\infty} \frac{t}{t^2+k^2} \leq \sum_{1\leq k\leq 2t } \frac{t}{t^2}+ \sum_{k>2t}  \frac{t}{k^2}\leq 2+ t\min(\frac{1}{2t-1}, 1+\frac{1}{2t})\leq \frac{7}{2}.
$$
Thus, $\alpha_0<+\infty$.\\
The case of $\beta_0$ is more delicate. We call $\Lambda(t)$ the set of integers $k$ such that $\vert t-k^2\vert \leq \frac{1}{4}(t+k^2)$. We have
$$ \sum_{k\notin \Lambda(t)}  \frac{t+k^2}{k^2+(t-k^2)^2}\leq 16 \sum_{k=1}^{+\infty} \frac{t+k^2}{t^2+k^4}\leq 16\,  (\frac{7}{2}+\frac{\pi^2}{6}).$$
We then must estimate  $ \sum_{k\in  \Lambda(t)}  \frac{t+k^2}{k^2+(t-k^2)^2}$.  On $\Lambda(t)$, we have $k^2\in (\frac{3}{5}t,\frac{5}{3}t)$, thus $\vert k-\sqrt t\vert<\frac{1}{3}\sqrt t$ and $\vert k^2-t\vert\geq \vert k-\sqrt t\vert \, \frac{5}{3}\sqrt t$. This gives
$$ \sum_{k\in  \Lambda(t)}  \frac{t+k^2}{k^2+(t-k^2)^2}\leq \frac{8}{3} \sum_{k=1}^\infty \min(\frac{5}{3}, \frac{9}{25 (k-\sqrt t)^2})\leq 5+\frac{8\pi^2}{25}.$$
 \end{proof}
 
 \subsection{Another proof of Lemma \ref{lemXY}: Energy-type estimate}
 
 We give another proof of Lemma \ref{lemXY} :
 
 \begin{proof}  We have, for $k\neq 0$,   $\left \vert  \vert\xi\vert^2-\frac{4\pi^2}{T^2}k^2+i\frac{2\pi}{T}k \right\vert\geq 
\frac{2\pi}{T} \vert k\vert$, and for $\vert\xi\vert> \frac{4\pi}{T} \vert k\vert$, $\left \vert  \vert\xi\vert^2-\frac{4\pi^2}{T^2}k^2+i\frac{2\pi}{T}k \right\vert\geq 
\frac{3}{4} \vert\xi\vert^2.$ Thus, it is straightforward that the solution $(\vec u,\vec E,\vec B)$ of system (\ref{syst}) satisfies  
$$\int_0^T \|\vec E_f\|_{H^{1/2+\delta}}^2+\|\vec B_f\|_{H^{1/2+\delta}}^2\, dt\leq C \int_0^T \|\vec G_f\|_{H^{1/2+\delta}}^2+\|\vec H_f\|_{H^{1/2+\delta}}^2\, dt.$$

Now, if we assume that  $ \vec G_f$ and $\vec H_f$ are trigonometric polynomials with respect to the time variable and with values, in $H^{3/2+\delta}$, we find that $\vec E_f$ and $\vec B_f$ are in $\mathcal{C}_{\rm per} H^{1/2+\delta}$ and that $\partial_t \vec E_f$ and $\partial_t\vec f$ belong to $L^2_{\rm per} H^{\frac{1}{2}+\delta}$.. Moreover, we have, writing $\vec{ \mathcal{E}}= (Id-\Delta)^{\frac{1}{4}+\frac{\delta}{2}} \vec E_f$ and $\vec{\mathcal{B}}= (Id-\Delta)^{\frac{1}{4}+\frac{\delta}{2}} \vec B_f$, 
\begin{equation*}\begin{split} \frac{d}{dt}\left(\frac{ \|\vec E_f\|_{H^{\frac{1}{2}+\delta}}^2+  \|\vec B_f\|_{H^{\frac{1}{2}+\delta}}^2}{2}\right) &  =\int_{\mathbb{R}^3} \partial_t \vec{ \mathcal{E}}.\vec{ \mathcal{E}}+\partial_t\vec{ \mathcal{B}}.\vec{ \mathcal{B}}\, dx
\\   =   \int_{\mathbb{R}^3}   \vec{ \mathcal{E}}.( (Id-\Delta)^{\frac{1}{4}+\frac{\delta}{2}} G_f - \vec{ \mathcal{E}}+\vec\nabla\wedge \vec{ \mathcal{B}})&+\vec{ \mathcal{B}}.( (Id-\Delta)^{\frac{1}{4}+\frac{\delta}{2}} \vec H_f-\vec\nabla\wedge \vec{ \mathcal{E}})\, dx
\\   =   \int_{\mathbb{R}^3}   \vec{ \mathcal{E}}.( (Id-\Delta)^{\frac{1}{4}+\frac{\delta}{2}} G_f - \vec{ \mathcal{E}})&+\vec{ \mathcal{B}}. (Id-\Delta)^{\frac{1}{4}+\frac{\delta}{2}} \vec H_f \, dx
\\  \leq \|\vec E_f\|_{H^{1/2+\delta}}\|\vec G_f\|_{H^{1+2+\delta}}&+ \|\vec B_f\|_{H^{1/2+\delta}}\|\vec H_f\|_{H^{1+2+\delta}}
\end{split}\end{equation*}
This gives, for $-T\leq t_0\leq0\leq t\leq T$
\begin{equation*}\begin{split}    \|\vec E_f(t,.)\|_{H^{\frac{1}{2}+\delta}}^2+  \|\vec B_f(t,.)\|_{H^{\frac{1}{2}+\delta}}^2  & \leq    \|\vec E_f(t_0,..)\|_{H^{\frac{1}{2}+\delta}}^2+  \|\vec B_f(t_0,.)\|_{H^{\frac{1}{2}+\delta}}^2 \\
& +2\int_{-T}^T \|\vec E_f\|_{H^{1/2+\delta}}\|\vec G_f\|_{H^{1+2+\delta}} \, ds\\&+2\int_{-T}^T \|\vec B_f\|_{H^{1/2+\delta}}\|\vec H_f\|_{H^{1+2+\delta}}\, ds
\end{split}\end{equation*}
Integrating with respect to $t_0$, we find
  \begin{equation*}\begin{split}    \|\vec E_f(t,.)\|_{H^{\frac{1}{2}+\delta}}^2+  \|\vec B_f(t,.)\|_{H^{\frac{1}{2}+\delta}}^2  & \leq  \frac{1}{T}\int_0^T  \|\vec E_f(t_0,..)\|_{H^{\frac{1}{2}+\delta}}^2+  \|\vec B_f(t_0,.)\|_{H^{\frac{1}{2}+\delta}}^2\, dt_0 \\
& +4\int_0^T \|\vec E_f\|_{H^{1/2+\delta}}\|\vec G_f\|_{H^{1+2+\delta}} \, ds\\&+4\int_0^T \|\vec B_f\|_{H^{1/2+\delta}}\|\vec H_f\|_{H^{1+2+\delta}}\, ds
\end{split}\end{equation*}
and finally
$$  \|\vec E_f(t,.)\|_{H^{\frac{1}{2}+\delta}}^2+  \|\vec B_f(t,.)\|_{H^{\frac{1}{2}+\delta}}^2  \leq C \int_0^T  \|\vec G_f\|_{H^{1+2+\delta}} ^2+\|\vec H_f\|_{H^{1+2+\delta}} ^2\, ds).$$
We then conclude the proof of the lemma by a density argument.
 \end{proof}
 
 $\ $
 
\begin{rem}
In the proof of Theorem \ref{theoNSM}, we actually show that $\vec U_f\in \mathcal{C}_{\rm per}\dot H^{1/2}\cap L^2_{\rm per} \dot H^{3/2+\delta}$ and, $\vec E\in \mathcal{C}_{\rm per}H^{1/2+\delta}$ and $\vec B\in \mathcal{C}_{\rm per}H^{1/2+\delta}$ while $\vec U_0\in\dot B^{1/2}_{2,\infty}\cap \dot H^{3/2+\delta}$. Thus the most inconvenient term to deal with is thus the mean value $\vec U_0=\frac{1}{T}\int_0^T \vec u(t,.)\, dt$.
\end{rem} 
 
 \subsection{Proof of Theorem \ref{MAIN1}}
In this part, we only sketch the proof  Theorem \ref{MAIN1} as the steps are basically similar to those of Theorem \ref{theoNSM}.
First, we shall adjust the previous spaces and define  
\begin{itemize}
\item $(\vec u,\vec E,\vec B)\in \tilde{\mathbb{X}} $ if   \begin{enumerate}
 \item $\vec u $ belongs to $\tilde L^\infty_{\rm per} \dot B^{\frac{1}{2}}_{2,(\infty,1)}\cap\tilde L^2_{\rm per} \dot  B^\frac{3}{2}_{2,(\infty,1)}$
 \item $\vec E $ and $\vec B $ belong  to $ \tilde L^\infty_{\rm per} H^\frac{1}{2}$.
 \end{enumerate}
\item $(\vec F,\vec G,\vec H)\in \tilde{\mathbb{Y}}$ if 
\begin{enumerate}
\item $\vec F $ belongs to $\tilde L^2_{\rm per} \dot B^{-\frac{1}{2}} _{2,(\infty,1)}$
\item  the mean value  $\vec F_0=\frac{1}{T}\int_0^T \vec F (t,.)\, dt$   belongs  to $ \dot B^{-\frac{3}{2}}_{2,(\infty,1)}$
 \item $\vec G $ belongs to $  L^2_{\rm per} H^\frac{1}{2}$ 
 \item the mean value  $\vec G_0=\frac{1}{T}\int_0^T \vec G (t,.)\, dt$   belongs to $\dot H^{-1}$
\item $\vec H $ is divergence-free ($\text{\rm div }\vec H =0$)  and  $\vec H $ belongs to $  L^2_{\rm per} H^\frac{1}{2}$  
\item the mean value   $\vec H_0=\frac{1}{T}\int_0^T \vec H (t,.)\, dt$ belongs to $\dot H^{-2}$. \end{enumerate}
\end{itemize}
Lemma \ref{law prod lem} can be extended to the following result in the case of critical Besov spaces. Again, it is sufficient to treat point-wise estimates (at fixed time).
\begin{lemma}\label{law prod lem Bes}
\begin{eqnarray}
H^\frac12\times H^\frac12 \hookrightarrow \dot B^{-\frac{1}{2}} _{2,1} \hookrightarrow \dot B^{-\frac{1}{2}} _{2,(\infty,1)}\\
\left(\dot B^{\frac{1}{2}} _{2,(\infty,1)}\cap \dot B^{\frac{3}{2}} _{2,(\infty,1)}\right)\times H^\frac12\hookrightarrow H^\frac12\\
\left(\dot B^{\frac{1}{2}} _{2,(\infty,1)}\cap \dot B^{\frac{3}{2}} _{2,(\infty,1)}\right)\times \left(\dot B^{\frac{1}{2}} _{2,(\infty,1)}\cap \dot B^{\frac{3}{2}} _{2,(\infty,1)}\right)\hookrightarrow \dot B^{\frac{1}{2}} _{2,(\infty,1)}.
\label{NS est}
\end{eqnarray}
\end{lemma}

\begin{proof}
The first product is classical and we omit it here.  To prove the other two, we first observe that
$$
\left(\dot B^{\frac{1}{2}} _{2,(\infty,1)}\cap \dot B^{\frac{3}{2}} _{2,(\infty,1)}\right)\hookrightarrow L^\infty.
$$
Let $u\in \dot B^{\frac{1}{2}} _{2,(\infty,1)}\cap \dot B^{\frac{3}{2}} _{2,(\infty,1)}$, and $B\in H^\frac12$. 
We begin by estimating the term $T_uB$ in the para-product. We have
$$
\|\Delta_qT_uB\|_{L^2}\leq \|u\|_{L^\infty}\|\Delta_qB\|_{L^2}\lesssim\|u\|_{\dot B^{\frac{1}{2}} _{2,(\infty,1)}\cap \dot B^{\frac{3}{2}} _{2,(\infty,1)}}\|\Delta_qB\|_{L^2}
$$
giving 
$$
\|T_uB\|_{H^\frac12}\lesssim \|u\|_{\dot B^{\frac{1}{2}} _{2,(\infty,1)}\cap \dot B^{\frac{3}{2}} _{2,(\infty,1)}} \|B\|_{H^\frac12}.
$$
When $k\leq0$, by Bernstein's lemma we have 
\begin{eqnarray*}
2^\frac k2\sum_{j\leq k}2^{\frac{3j}2}\|\Delta_jB\|_{L^2}\|\Delta_ku\|_{L^2}&\leq &\sum_{j\leq k}2^{\frac j2}\|\Delta_jB\|_{L^2}2^{j-k}  2^{\frac{3k}2}  \|\Delta_ku\|_{L^2}\\
&\lesssim&
\|u\|_{\dot B^{\frac{1}{2}} _{2,(\infty,1)}} \sum_{j\leq k}2^{\frac j2}\|\Delta_jB\|_{L^2}2^{j-k}
\end{eqnarray*}
For $k\geq1$, we decompose further as follows
\begin{eqnarray}
\sum_{j\leq k}2^{\frac{3j}2}\|\Delta_jB\|_{L^2}\|\Delta_ku\|_{L^2}&\leq &\sum_{j\leq0}(\cdot)+\sum_{0\leq j\leq k}(\cdot)
\end{eqnarray}
and estimate the terms as below
$$
2^\frac k2\sum_{j\leq0}(\cdot)\lesssim \|B\|_{H^\frac12} 2^\frac k2\|\Delta_ku\|_{L^2}
$$
and
$$
2^\frac k2\sum_{0\leq j\leq k}(\cdot)\lesssim 2^\frac{3k}2\|\Delta_ku\|_{L^2} \sum_{0\leq j\leq k} 2^\frac j2\|\Delta_jB\|_{L^2}  2^{j-k} .
$$
Using Young's inequality, we conclude that
$$
\|T_Bu\|_{H^\frac12}\lesssim \|u\|_{\dot B^{\frac{1}{2}} _{2,(\infty,1)}\cap \dot B^{\frac{3}{2}} _{2,(\infty,1)}} \|B\|_{H^\frac12}.
$$
Finally we estimate the remainder term $R(u,B)$ only when $k\leq0$ because the case $k\geq1$ is easier. By Bernstein's lemma we have 
\begin{eqnarray*}
2^\frac k2\sum_{j\geq k-2}2^{\frac{3q}2}\|\Delta_jB\|_{L^2}\|\Delta_ju\|_{L^2}&\leq &2^{2k}\sum_{j\leq 0}\|\Delta_ju\|_{L^2}\|\Delta_jB\|_{L^2}\\
&+&2^{2k}\sum_{k-2\leq j\leq0}\|\Delta_jB\|_{L^2}  \|\Delta_ju\|_{L^2}\\
&\lesssim&
\|u\|_{\dot B^{\frac{1}{2}} _{2,(\infty,1)}} \|B\|_{L^2}2^{2k}\\&+&\sum_{j\geq k-2}2^{\frac{3j}2}\|\Delta_ju\|_{L^2} 2^\frac j2 \|\Delta_jB\|_{L^2}2^{2(k-j)}.
\end{eqnarray*}
Thus we conclude that 
$$
\|R(u,B)\|_{H^\frac12}\lesssim \|u\|_{\dot B^{\frac{1}{2}} _{2,(\infty,1)}\cap \dot B^{\frac{3}{2}} _{2,(\infty,1)}} \|B\|_{H^\frac12},
$$
as required. The proof of estimate \eqref{NS est} is similar. As before we have
$$
\|\Delta_qT_uu\|_{L^2}\leq \|u\|_{L^\infty}\|\Delta_qu\|_{L^2}\lesssim\|u\|_{\dot B^{\frac{1}{2}} _{2,(\infty,1)}\cap \dot B^{\frac{3}{2}} _{2,(\infty,1)}}\|\Delta_qu\|_{L^2}
$$
giving 
$$
\|T_uu\|_{\dot B^\frac12_{2,(\infty,1)}}\lesssim \|u\|_{\dot B^{\frac{1}{2}} _{2,(\infty,1)}\cap \dot B^{\frac{3}{2}} _{2,(\infty,1)}} \|u\|_{\dot B^{\frac{1}{2}} _{2,(\infty,1)}}.
$$
Next we estimate the remainder term only for $k\leq0$, by Bernstein's lemma we have 
\begin{eqnarray*}
2^\frac k2\sum_{j\geq k-2}2^{\frac{3q}2}\|\Delta_ju\|_{L^2}\|\Delta_ju\|_{L^2}&\leq &2^{2k}\sum_{j\leq 0}\|\Delta_ju\|_{L^2}\|\Delta_ju\|_{L^2}\\
&+&2^{2k}\sum_{k-2\leq j\leq0}\|\Delta_ju\|_{L^2}  \|\Delta_ju\|_{L^2}\\
&\lesssim&
\|u\|_{\dot B^{\frac{1}{2}} _{2,(\infty,1)}}^2 2^{2k}\\&+&\sum_{j\geq k-2}2^{\frac{3j}2}\|\Delta_ju\|_{L^2} 2^\frac j2 \|\Delta_ju\|_{L^2}2^{2(k-j)}.
\end{eqnarray*}
Consequently, we have
$$
\|R(u,u)\|_{H^\frac12}\lesssim \|u\|_{\dot B^{\frac{1}{2}} _{2,(\infty,1)}\cap \dot B^{\frac{3}{2}} _{2,(\infty,1)}}^2.
$$
as desired.
\end{proof}
Now, we give a result equivalent to Lemma \ref{lemXY}.

\begin{lemma}\label{lem tildeXY}$\ $\\
 Let  $(\vec F,\vec G,\vec H)\in \tilde{\mathbb{Y}}$.  Then the time-periodic solution $\vec\Gamma:=(\vec u,\vec E,\vec B)$ of the system
 \begin{equation} \label{syst2}
\left\{\begin{split} \partial_t \vec  u-  \Delta\vec u&= \vec F \\  \partial_t\vec E-\vec\nabla\wedge\vec B +\vec  E &=\vec G
\\ \partial_t\vec B +\vec\nabla\wedge \vec E &=\vec H 
\\   \text{\rm  div }\vec  B &=0 
\end{split}\right.\end{equation}
 satisfies $(\vec u,\vec E,\vec B)\in\tilde{\mathbb{X}}$.
 \end{lemma}

\begin{proof}
We only estimate the solution $\hat U_k(\xi)= \frac{1}{\vert\xi\vert^2+i\frac{2\pi}{T}k}\hat F_k(\xi) $ in $\tilde L^2(\dot B^\frac32_{2,(\infty,1)})$, as all the other estimates are similar. First, we have
$$
\|\Delta_q\vec U\|_{L^2\left(0,T;L^2\right)}\lesssim \sqrt{\sum_{k\in\mathbb Z\setminus0} \frac{\|\Delta_q\vec F_k\|_{L^2}^2}{2^{4q}+\frac{k^2}{T^2}}}.
$$
Hence,
\begin{eqnarray*}
\sup_{q\leq0}2^\frac{3q}2\|\Delta_q\vec U\|_{L^2\left(0,T;L^2\right)}&\lesssim& \sup_{q\leq0}2^{-\frac q2}\sqrt{\sum_{k\in\mathbb Z}\|\Delta_q\vec F_k\|_{L^2}^2} \sqrt{\sup_{q\leq0}\sup_{k\in\mathbb Z\setminus0}\frac{2^{4q}}{2^{4q}+\frac{k^2}{T^2}}}\\
&\lesssim& \sup_{q\leq0}2^{-\frac q2}\sqrt{\sum_{k\in\mathbb Z\setminus0}\|\Delta_q\vec F_k\|_{L^2}^2}\lesssim\|F\|_{\tilde L^2(\dot B^{-\frac12}_{2,(\infty,1)})}
,
\end{eqnarray*}
and similarly
\begin{eqnarray*}
\sum_{q\geq1}2^\frac {3q}2\|\Delta_q\vec U\|_{L^2\left((0,T);L^2\right)}&\lesssim& \sum_{q\leq0}2^{-\frac q2}\sqrt{\sum_{k\in\mathbb Z\setminus0}\|\Delta_q\vec F_k\|_{L^2}^2} \sqrt{\sup_{q\leq0}\sup_{k\in\mathbb Z\setminus0}\frac{2^{4q}}{2^{4q}+\frac{k^2}{T^2}}}\\
&\lesssim& \sum_{q\geq1}2^{-\frac q2}\sqrt{\sum_{k\in\mathbb Z\setminus0}\|\Delta_q\vec F_k\|_{L^2}^2}\lesssim\|F\|_{\tilde L^2(\dot B^{-\frac12}_{2,(\infty,1)})}.
\end{eqnarray*}
\end{proof}

\section{Asymptotic stability}
The purpose of this section is to prove Theorem \ref{MAIN2}. In the sequel, we omit the $\vec{ }$ symbol in order to alleviate the notation, since we will use also $\hat{ }$ for the Fourier transform, $\tilde{}$ etc.\\
Denote by $\Gamma_{\text{per}}$ the $T$-periodic small solution of \eqref{eqns} given by Theorem \ref{MAIN1}. We decompose $\bar\Gamma$ the solution of \eqref{eqns} as
\begin{eqnarray}
\nonumber
\bar\Gamma&:=&\Gamma_{\text{per}}+\Gamma_{\rm err}
\label{decomp}
\end{eqnarray}
where the ``error" term $\Gamma_{\rm err}$ is further decomposed as $\Gamma_{\rm err}=e^{t{\mathcal A}}\Gamma^0_{\rm err}+\Gamma$. We assume that the initial data $\Gamma^0_{\rm err}$ is small in the space 
$   \dot{\mathcal B}^\frac12_{2,(\infty,1)}\times H^\frac12\times H^\frac12  $, and $\Gamma^0=0$. 
It is easy to see that $\Gamma_{\rm err}$ solves the following integral equation

$$
\Gamma_{\rm err}(t)=e^{t{\mathcal A}}\Gamma^0_{\rm err}+\int_0^te^{(t-t'){\mathcal A}}{\mathcal N}(\Gamma_{\rm err}(t'))\;dt',
$$
with
$$
{\mathcal A}=\left(
\begin{array}{ccc}
\Delta&0&0\\
0&-I&\nabla\wedge\cdot \\0&-\nabla\wedge\cdot &0
\end{array}
\right)
$$
and the three components of the nonlinearity ${\mathcal N}=({\mathcal N}_1,{\mathcal N}_2 ,{\mathcal N}_3 )$ are 
\begin{eqnarray*}
{\mathcal N}_1=\mathbb{P}\Big(-\nabla(u\otimes u)&-&\nabla(u_{\rm per}\otimes u)-\nabla(u\otimes u_{\rm per})\\
 &+&E\times B+E\times B_{\rm per}+ E_{\rm per}\times  B\\
&+&(u\times B)\times B+(u\times B_{\rm per})\times B_{\rm per}\\
&+&[u_{\rm per} \times B]\times B_{\rm per}+[ u_{\rm per} \times B_{\rm per}]\times B\Big),
\end{eqnarray*}

$$
{\mathcal N}_2=u\times B+u\times B_{\rm per}+  u_{\rm per}\times B
$$
and ${\mathcal N}_3=0$, respectively. Observe that the nonlinear term is expressed only in terms of the periodic solution. The construction of $\Gamma$ follows a standard fixed point argument. \\
Let $B_{\delta}$ be the ball of the space ${\mathcal X}$ 
centred at  zero and with radius $\delta>0$ to be chosen. 
On that ball, define the map $\Phi$  as follows
\begin{eqnarray}
\nonumber
 \Phi:B_{\delta}\subset{\mathcal X}&\longrightarrow&{\mathcal X}\\
 \label{contraction map}
\Gamma&\mapsto& \Phi(\Gamma):=\int_0^te^{(t-t'){\mathcal A}}{\mathcal N}
(e^{t'{\mathcal A}}\Gamma^0_{\rm err}+\Gamma(t'))\;dt'.
\end{eqnarray}
Hence, the result of Theorem \ref{MAIN2} will be a consequence of the following proposition. 

\begin{proposition}
\label{contraction}
If $\|\Gamma^0_{\rm err}\|_{\dot {\mathcal B}^{\frac12}_{2,(\infty,1)}\times H^{\frac12}
\times H^{\frac12}}\leq \kappa \delta$, with $\delta>0$ and $\kappa>0$
sufficiently small, then the map $\Phi$ is a contraction on $B_{\delta}$. 
\end{proposition}
Indeed, admitting for now this proposition, Picard's theorem gives the existence of a fixed point
of the map $\Phi$, call it $\Gamma$. Then clearly $e^{t{\mathcal A}}\Gamma^0_{\rm err} + \Gamma(t)$ would be the desired solution of \eqref{eqns}.\\
In order to prove Proposition \ref{contraction}, we need a few preliminary lemmas. 
\subsection{Preliminary results}
We start with several preliminary lemmas. First, we prove the following result of parabolic regularity, in the spirit of \cite{BCD},  adapted to the spaces $\mathcal X$ and $\mathcal Y_1$ in the following way.

\begin{lemma}[Adapted maximum parabolic regularity]
\label{ghr}
Let $u$ be a smooth divergence free vector field solving
\begin{equation}
   \label{g u-f} 
   \left\{ \begin{array}{l} \partial_t u - \Delta
u + \nabla p
 = f \\
 u_{|t=0} = u^0, 
\end{array} \right. 
\end{equation}
on some time interval $[0,T]$. Then, we have
 \begin{equation*}
 \| u  \|_{\mathcal X_1\cap\tilde L^\infty(\dot B^\frac12_{2,(\infty,1)}) }   \lesssim    \| u^0 \|_{\dot B^{\frac12}_{2,(\infty,1)}} + 
\| f  \|_{\mathcal Y_1}.
\end{equation*}
\end{lemma}

\begin{proof}
From the start, for any $k\in\mathbb Z$, and $0\leq\varepsilon\leq1$ denote by
$$
\alpha_k:=\min(2^k,1),
$$
and define the norm 
$$
  M_{k,\varepsilon}(f):=\sup_{n\in\mathbb N}\;(2^{-k(\frac12+\varepsilon)}(n+1)^\frac{1-\varepsilon}2\|\Delta_kf\|_{L^2_t(n,n+1;L^2_x)}),
$$
to be used to control the low frequencies of $f$, and 
$$
\tilde M_{k,\varepsilon}(f):=\sup_{n\in\mathbb N}\;(2^{-\frac k2}(n+1)^\frac{1-\varepsilon}2\|\Delta_kf\|_{L^2_t(n,n+1;L^2_x)}),
$$ 
to control its high frequencies. Also, observe that for all $c>0$ we have
\begin{eqnarray}
\label{element est}
x^{{1-\varepsilon}}e^{-cx^2}\leq C(c,\varepsilon)\lesssim 1.
\end{eqnarray} 
Moreover, the following elementary estimate 
\begin{eqnarray*}
\int_0^Ae^{u^2}\;du\leq\frac{e^A-1}A
\end{eqnarray*} 
clearly implies the following one
\begin{eqnarray}
\label{integ est}
\sup_{t>0}t^\frac12e^{-\frac t2}\int_0^t\frac{e^{\frac s2}}{\sqrt s}\;ds\leq4
\end{eqnarray}
which will be of frequent use in the proofs of our linear estimates. In addition, the following estimate is classical and can be found, for example, in \cite{BCD}
\begin{eqnarray}
 \label{key dyadic heat est}
\|\Delta_ke^{t\Delta}u_0\|_{L^2}\lesssim e^{-ct2^{2k}}\|\Delta_ku_0\|_{L^2}.
\end{eqnarray}
From now on, we will `drop`` this constant $c$ by taking it always equals to one.\\
Duhamel's formula for the solution of \eqref{g u-f} gives 
$$
u(t)=e^{t\Delta}u_0+\int_0^te^{(t-s)\Delta}\mathbb{P} f(s)\;ds,
$$ 
where $\mathbb P$ is Leray's projection.

$\bullet$ First, we focus on the homogeneous solution $e^{t\Delta}u_0$. Multiplying \eqref{key dyadic heat est}
by $2^\frac k2$, taking the supremum in time and then summing in $k$ 
(and the supremum in $k$ for low frequencies), easily gives

{
\begin{eqnarray}
 \label{heat est1}
\tilde{\sup_{t>0}}\;\|e^{t\Delta}u_0\|_{\dot{\mathcal B}^\frac12_{2,(\infty,1)}}\lesssim 
\|u_0\|_{\dot{\mathcal B}^\frac12_{2,(\infty,1)}}.
\end{eqnarray}}
Now we focus on the norm giving the decay. From \eqref{key dyadic heat est} and for all $0\leq\varepsilon\leq1$, we have for any $k\leq0$

{
\begin{eqnarray}
\label{homo est1}
\nonumber
(t+1)^\frac{1-\varepsilon}2 2^{k(\frac32-\varepsilon)}\|\Delta_k e^{t\Delta}u_0\|_{L^2}
&\lesssim& \big(2^{2k}(t+1)\big)^\frac{1-\varepsilon}2 e^{-ct2^{2k}}  2^{\frac k2}\|\Delta_ku_0\|_{L^2}\\
\nonumber
&\lesssim& \big(2^{2k}(t+1)\big)^\frac{1-\varepsilon}2 e^{-(t+1)2^{2k}}  2^{\frac k2}\|\Delta_ku_0\|_{L^2}\\
&\lesssim& 
2^{\frac k2}\|\Delta_ku_0\|_{L^2},
\end{eqnarray}}
where we used \eqref{element est}. When $k\geq1$, we estimate as follows

\begin{eqnarray}
\label{homo est2}
\nonumber
(t+1)^\frac{1-\varepsilon}2 2^{\frac k2}\|\Delta_k e^{t\Delta}u_0\|_{L^2}
&\lesssim& (t+1)^\frac{1-\varepsilon}2 e^{-t2^{2k}}  2^{\frac k2}\|\Delta_ku_0\|_{L^2}\\
\nonumber
&\lesssim& (t+1)^\frac{1-\varepsilon}2 e^{-(t+1)}  2^{\frac k2}\|\Delta_ku_0\|_{L^2}\\
&\lesssim& 
2^{\frac k2}\|\Delta_ku_0\|_{L^2}.
\end{eqnarray}
Obviously, \eqref{homo est1} and \eqref{homo est2} give 
$$
\tilde{\sup_{t>0}}\;(t+1)^\frac{1-\varepsilon}2
\|e^{t\Delta}u_0\|_{\dot{\mathcal B}^{\frac32-\varepsilon,\frac12}_{2,(\infty,1)}}\lesssim 
\|u_0\|_{\dot{\mathcal B}^{\frac12}_{2,(\infty,1)}}.
$$
Again thanks to \eqref{key dyadic heat est}, we have
\begin{eqnarray*}
\|\Delta_k e^{t\Delta}u_0\|_{L^2(n,n+1,L^2_x)}&\lesssim& \frac{\sqrt{  e^{-2n2^{2k}}-e^{-2(n+1)2^{2k}}   }}{2^k}\|\Delta_ku_0\|_{L^2}\\
&\lesssim& 2^{-k}\alpha_k e^{-n2^{2k}}\|\Delta_ku_0\|_{L^2},
\end{eqnarray*}
so that
$$
(n+1)^{\frac{1-\varepsilon}2} 2^{\frac32k}\|\Delta_ku\|_{L^2(n,n+1,L^2_x)}
\lesssim (n+1)^{\frac{1-\varepsilon}2} e^{-(n+1)} 2^{\frac k2}\|\Delta_ku_0\|_{L^2}\lesssim 2^{\frac k2}\|\Delta_ku_0\|_{L^2},
$$
and consequently, one obtains 
$$
\tilde{\sup_{n}}\;(n+1)^{\frac{1-\varepsilon}2}
\|e^{t\Delta}u_0\|_{\tilde L^2(n,n+1;\;\dot{\mathcal B}^\frac32_{2,(\infty,1)})}\lesssim 
\|u_0\|_{\dot{\mathcal B}^{\frac12}_{2,(\infty,1)}}
$$
as desired.

$\bullet$ Second, we focus on the non-homegenous solution $v(t):=\int_0^te^{(t-s)\Delta}\mathbb P f(s)\;ds$, where $\mathbb P$ is Leray's projection.
Fix $t>0$ and let $[t]$ be its integer part.  Decompose

\begin{eqnarray}
\nonumber
\|\Delta_kv(t)\|_{L^2_x}&\lesssim& \sum_{n=0}^{[t]-1} \int_n^{n+1}
e^{-(t-s)2^{2k}}\|\Delta_k f(s)\|_{L^2}\;ds+\int_{[t]}^te^{-(t-s)2^{2k}}\|\Delta_k f(s)\|_{L^2}\;ds.
\end{eqnarray}
Notice that if $0\leq t<1$, then only the last term shows up in the last inequality. We start by estimating 
$v$ in $\tilde L^\infty(\dot{\mathcal B}^\frac12_{2,(\infty,1)})$. \\
For any $n\leq [t]-1$, we have 
\begin{eqnarray}
\nonumber
\int_n^{n+1} e^{-(t-s)2^{2k}}\|\Delta_k f(s)\|_{L^2}\;ds&\lesssim& 
\Big(\max_{n}(n+1)^\frac{1-\varepsilon}2\|\Delta_kf\|_{L^2(n,n+1,L^2_x)}\Big)
\frac{ e^{-(t-n-1)2^{2k}} \alpha_k    }{(n+1)^\frac{1-\varepsilon}22^{k}}\\
&\lesssim& M_{k,\varepsilon}(f)2^{-k(\frac12-\varepsilon)}\alpha_ke^{-t2^{2k}}\frac{ e^{(n+1)2^{2k}}     }{(n+1)^\frac{1-\varepsilon}2}.
 \label{sum est0}
\end{eqnarray}
where we used $\sqrt{1-e^{-22^{2k}}}\leq\alpha_k$ in the first above estimate. Moreover,
\begin{eqnarray}
 \nonumber
\int_{[t]}^te^{-(t-s)2^{2k}}\|\Delta_k f(s)\|_{L^2}\;ds&\lesssim& 
\Big(\max_{n}(n+1)^\frac{1-\varepsilon}2\|\Delta_kf\|_{L^2(n,n+1,L^2_x)}\Big)
\frac{\sqrt{ 1 -e^{-2(t-[t])2^{2k}}   }}{  (t+1)^\frac{1-\varepsilon}22^{k}   }\\
&\lesssim& M_{k,\varepsilon}(f)2^{-k(\frac12-\varepsilon)}\frac{\min\big(1,\sqrt{2(t-[t])2^{2k}}\big)}{(t+1)^\frac{1-\varepsilon}2}\label{sum est1}\\
&\lesssim& M_{k,\varepsilon}(f)2^{-k(\frac12-\varepsilon)}.
\nonumber
\end{eqnarray}
Thus,

\begin{eqnarray}
 \label{low}
\nonumber
2^{\frac k2}\|\Delta_kv(t)\|_{L^2_x}&\lesssim& M_{k,\varepsilon}(f)\Big(\alpha_k2^{k\varepsilon}e^{-t2^{2k}}\sum_{n=0}^{[t]-1}\frac{ e^{(n+1)2^{2k}}     }{(n+1)^\frac{1-\varepsilon}2}+2^{k\varepsilon}\Big).
\end{eqnarray}
Now, when $k\leq0$, to estimate the sum, we distinguish two cases. In the case $[t]2^{2k}\leq1$ we use 

\begin{eqnarray}
\label{sum est2}
\alpha_k2^{k\varepsilon} e^{-t2^{2k}}\sum_{n=0}^{[t]-1}\frac{ e^{(n+1)2^{2k}}     }{(n+1)^\frac{1-\varepsilon}2}\lesssim \alpha_k2^{k\varepsilon} [t]^\frac{1+\varepsilon}2\lesssim (2^{2k}[t])^\frac{1+\varepsilon}2\lesssim1.
\end{eqnarray}
In the case, $[t]2^{2k}\geq1$, we use \eqref{integ est} and estimate in the following way

\begin{eqnarray}
\label{sum est3}
\sum_{n=0}^{[t]-1}\frac{ e^{(n+1)2^{2k}}     }{(n+1)^\frac{1-\varepsilon}2}\leq [t]^\frac\varepsilon2 \sum_{n=0}^{[t]-1}\frac{ e^{(n+1)2^{2k}}     }{\sqrt{n+1}}\lesssim 2^{-k}\frac{     e^{[t]2^{2k}}    }{   (2^{2k}[t])^\frac{1-\varepsilon}2}.
\end{eqnarray}
Consequently, in both cases, we have
$$
\alpha_k2^{k\varepsilon}e^{-t2^{2k}}\sum_{n=0}^{[t]-1}\frac{ e^{(n+1)2^{2k}}     }{(n+1)^\frac{1-\varepsilon}2}\lesssim1, 
$$
and therefore
$$
\sup_{k\leq0}\sup_{t>0}2^\frac k2\|\Delta_kv(t)\|_{L^2}\lesssim \sup_{k\leq0}M_{k,\varepsilon}(f)
$$
Next, we treat frequencies $k\geq1$. Arguing exactly as above, we obtain

\begin{eqnarray}
\nonumber
\int_n^{n+1} e^{-(t-s)2^{2k}}\|\Delta_k f(s)\|_{L^2}\;ds + \int_{[t]}^te^{-(t-s)2^{2k}}\|\Delta_k f(s)\|_{L^2}\;ds&\lesssim& 
\tilde M_{k,\varepsilon}(f)2^{-\frac k2}
\end{eqnarray}
yielding

$$
\sum_{k\geq0}\sup_{t>0}2^\frac k2\|\Delta_kv(t)\|_{L^2}\lesssim \sum_{k\geq0}\sup_{n\in\mathbb N}(n+1)^\frac{1-\varepsilon}22^{-\frac k2}\|\Delta_kf\|_{L^2(n,n+1;L^2)}.
$$
Whence
$$
\tilde{\sup_{t>0}}\;\|u(t)\|_{\dot{\mathcal B}^\frac12_{2,(\infty,1)}}\leq \|f\|_{\mathcal Y_1}.
$$
This shows the estimate in $\tilde L^\infty({\dot{\mathcal B}^\frac12_{2,(\infty,1)}})$. Next we will estimate $v(t)$ in $\mathcal X_1$.\\ 
Thanks to \eqref{sum est0} and \eqref{sum est1} we have

\begin{eqnarray*}
(t+1)^\frac{1-\varepsilon}2 2^{k(\frac12-\varepsilon)}\alpha_k \|\Delta_kv(t)\|_{L^2_x}
&\lesssim&(t+1)^{\frac{1-\varepsilon}2}
{M_{k,\varepsilon}(f)}\alpha_k^2
e^{-t2^{2k}}\sum_{n=0}^{[t]-1}
\frac{e^{(n+1)2^{2k}}      }    {(n+1)^\frac{1-\varepsilon}2}\\
&+&
{M_{k,\varepsilon}(f)}\alpha_k.
\end{eqnarray*}
In the case $[t]2^{2k}\leq1$ we use \eqref{sum est2} to conclude that

$$
(t+1)^{\frac{1-\varepsilon}2}{M_{k,\varepsilon}(f)}\alpha_k^2
e^{-t2^{2k}}\sum_{n=0}^{[t]-1}
\frac{e^{(n+1)2^{2k}}      }    {(n+1)^\frac{1-\varepsilon}2}\lesssim M_{k,\varepsilon}(f)\alpha_k^2(t+1)\lesssim M_{k,\varepsilon}(f).
$$
In the case $[t]2^{2k}\geq1$ we use \eqref{sum est3} to conclude that

$$
(t+1)^{\frac{1-\varepsilon}2}{M_{k,\varepsilon}(f)}\alpha_k^2
e^{-t2^{2k}}\sum_{n=0}^{[t]-1}
\frac{e^{(n+1)2^{2k}}      }    {(n+1)^\frac{1-\varepsilon}2}\lesssim M_{k,\varepsilon}(f)\alpha_k^22^{k(\varepsilon-2)}\lesssim M_{k,\varepsilon}(f).
$$
Taking the supremum in $t$ and then the supremum in $k\leq0$ gives

$$
\tilde{\sup_{t>0}}\;(t+1)^\frac{1-\varepsilon}2\|v^<(t)\|_{\dot{\mathcal B}^{\frac32-\varepsilon}_{2,\infty}}\lesssim 
\|f\|_{\mathcal Y_1}.
$$
Now we consider frequencies $k\geq1$. We have
\begin{eqnarray*}
(t+1)^\frac{1-\varepsilon}2 2^{\frac k2} \|\Delta_kv(t)\|_{L^2_x}
&\lesssim&(t+1)^\frac{1-\varepsilon}2
{\tilde M_{k,\varepsilon}(f)}\alpha_k
e^{-t2^{2k}}\sum_{n=0}^{[t]-1}
\frac{e^{(n+1)2^{2k}}      }    {(n+1)^\frac{1-\varepsilon}2}\\
&+&{M_{k,\varepsilon}(f)},
\end{eqnarray*}
and in the case $[t]2^{2k}\leq1$ we use \eqref{sum est2} to conclude that

$$
(t+1)^{\frac{1-\varepsilon}2}{\tilde M_{k,\varepsilon}(f)}
e^{-t2^{2k}}\sum_{n=0}^{[t]-1}
\frac{e^{(n+1)2^{2k}}      }    {(n+1)^\frac{1-\varepsilon}2}\lesssim\tilde M_{k,\varepsilon}(f)(t+1)\lesssim 2^{-2k}M_{k,\varepsilon}(f).
$$
In the case $[t]2^{2k}\geq1$ we use \eqref{sum est3} to conclude that

$$
(t+1)^{\frac{1-\varepsilon}2}{\tilde M_{k,\varepsilon}(f)}
e^{-t2^{2k}}\sum_{n=0}^{[t]-1}
\frac{e^{(n+1)2^{2k}}      }    {(n+1)^\frac{1-\varepsilon}2}\lesssim\tilde M_{k,\varepsilon}(f)2^{-k}\lesssim\tilde M_{k,\varepsilon}(f).
$$
Taking the supremum in $t$ and then summing in $k\geq0$ gives

$$
\tilde{\sup_{t>0}}\;(t+1)^\frac{1-\varepsilon}2\|v^>(t)\|_{\dot{\mathcal B}^{\frac12}_{2,1}}\lesssim 
\|f\|_{\mathcal Y_1}.
$$
In conclusion, we have shown that

$$
\tilde{\sup_{t>0}}\;(t+1)^\frac{1-\varepsilon}2\|v(t)\|_{\dot{\mathcal B}^{\frac32-\varepsilon,\frac12}_{2,(\infty,1)}}\lesssim 
\|f\|_{\mathcal Y_1}.
$$
Finally, we estimate $v$ in $\tilde\sup_n(n+1)^\frac{1-\varepsilon}2\|v\|_{L^2(n,n+1;\dot B^\frac32_{2,(\infty,1)})}$. Fix an integer $N$. For all $N\leq t< N+1$, arguing as before we obtain

\begin{eqnarray*}
\|\Delta_kv(t)\|_{L^2_x}&\lesssim&\tilde M_{k,\varepsilon}(f)\alpha_k
2^{-\frac k2}\sum_{n=0}^{N-1} \frac{e^{(n+1-t)2^{2k}}      }    {(n+1)^\frac{1-\varepsilon}2}     +\tilde M_{k,\varepsilon}(f)\alpha_k2^{-\frac k2}
\frac{\min\big(1,\sqrt{2(t-N)2^{2k}}\big)}{(N+1)^\frac{1-\varepsilon}2}\\
&\lesssim& M_{k,0}(f)\alpha_k2^{-\frac k2}\Big(\int_0^{[t]}   \frac{e^{-(t-s)2^{2k}}}{\sqrt s}\;ds+\frac{e^{(N-t)2^{2k}}}{\sqrt{N+1}}\Big)\\
&\lesssim& M_{k,0}(f)\alpha_k2^{-\frac k2}\Big(1+\frac{e^{(N-t)2^{2k}}}{\sqrt{N+1}}\Big)
\end{eqnarray*}
and

\begin{eqnarray*}
\|\Delta_kv(t)\|_{L^2(N,N+1,L^2_x)}&\lesssim&\tilde M_{k,\varepsilon}(f)\alpha_k^2
2^{-\frac32k}\sum_{n=0}^{N-1} \frac{e^{(n+1-N)2^{2k}}      }    {(n+1)^\frac{1-\varepsilon}2} 
\end{eqnarray*}
In the case $[t]2^{2k}\leq1$, we control the sum using \eqref{sum est2} to end up with

\begin{eqnarray*}
\|\Delta_kv(t)\|_{L^2(N,N+1,L^2_x)}&\lesssim& 2^{-\frac32k}\tilde M_{k,\varepsilon}(f)\alpha_k^2 N^\frac{1+\varepsilon}2.
\end{eqnarray*}
This leads to 

$$
(N+1)^\frac{1+\varepsilon}22^{\frac32k}\|\Delta_kv(t)\|_{L^2(N,N+1,L^2_x)}\lesssim \tilde M_{k,\varepsilon}(f)\alpha_k^2 N\lesssim \tilde M_{k,\varepsilon}(f)
$$
In the case $[t]2^{2k}\geq1$, we estimate the sum using \eqref{sum est3} to obtain
\begin{eqnarray*}
\|\Delta_kv(t)\|_{L^2(N,N+1,L^2_x)}&\lesssim& 2^{-\frac32k}\tilde M_{k,\varepsilon}(f)\alpha_k^2 \frac{2^{-k}}{(2^{2k}N)^\frac{1-\varepsilon}2},
\end{eqnarray*}
yielding estimates

$$
(N+1)^\frac{1-\varepsilon}22^{\frac32k}\|\Delta_kv(t)\|_{L^2(N,N+1,L^2_x)}\lesssim \tilde M_{k,\varepsilon}(f)\alpha_k^2 2^{(\varepsilon-2)k}\lesssim \tilde M_{k,\varepsilon}(f),
$$
and 
$$
\tilde{\sup_{N}}\;(N+1)^\frac{1-\varepsilon}2\|v(t)\|_{\tilde L^2(N,N+1;\;\dot{\mathcal B}^\frac32_{2,(\infty,1)})}\lesssim 
\|f\|_{\mathcal Y_1}
$$
as desired. This completes the proof of the Lemma. 

\end{proof}


Now we focus on Maxwell's equations.  
The following result quantifies a weak form of decay for the electromagnetic field $(E,B)$.

\begin{lemma}
\label{pr1}
Let $G\in \mathcal Y_1$, 
and $(E,B)$ be a smooth solution of 

\begin{equation}
   \label{Maxwell} 
   \left\{ \begin{array}{l} \partial_tE-\nabla\wedge B+E=G\\
\partial_tB+\nabla\wedge E=0\\
(E,B)_{|t=0} = (E^0,B^0), 
\end{array} \right. 
\end{equation}
on some time interval $[0,T]$. 
Then, the following estimate holds (with constants independent of $T$)
\begin{equation}
\label{energy-bis}
\|E\|_{\mathcal X_2}+
\|B\|_{\mathcal X_3}
 \lesssim  \|(E^0,B^0)\|_{ H^\frac12}+
\|G\|_{\mathcal Y_2}.
\end{equation}
\end{lemma}

\begin{proof}
First, we recall that in \cite{GIM}, the following estimate was proven.
\begin{equation}
\label{energy-bis-bis}
\|E\|_{{\tilde L}^\infty_T  H^{\frac12}\cap 
L^2_T \dot H^{\frac12}}+
\|B\|_{{\tilde L}^\infty_T  H^{\frac12}\cap L^2_T\dot H^{1,\frac12}}
 \lesssim  \|(E_0,B_0)\|_{ H^{\frac12}}+
\|G\|_{L^2_T  H^{\frac12}}
\end{equation}
with a constant independent of $T$. Hence, it will be sufficient to prove that

\begin{equation}
\label{decay}
\tilde{\sup_{t>0}}t^\frac12\;\|E(t)\|_{\dot H^\frac12}+
\tilde{\sup_{t>0}}t^\frac12\;\|B(t)\|_{\dot H^{1,\frac12}}
 \lesssim  \|(E_0,B_0)\|_{ H^{\frac12}}+
\|G\|_{L^2_T  H^{\frac12}}.
\end{equation}

We decompose $E=E_\sigma+E_\nabla$ into its divergence free component $E_\sigma$, and irrotational component $E_\nabla$. It is easy to check that $B$, $E_\sigma$ and $E_\nabla$ solve

\begin{equation}
\label{eq B}
\partial_t^2 B-\Delta B+\partial_t B=-\nabla \times G, 
\end{equation}

\begin{equation}
\label{eq E sigma}
 \partial_t B+\nabla\times E_\sigma=0
\end{equation}
 and

\begin{equation}
\label{eq E grad}
 \partial_t E_\nabla+E_\nabla=G_\nabla
\end{equation}
respectively. Thanks to the Fourier transform, and the spectral analysis given in the Appendix we have the following representation formula for $B$.

\begin{equation}
\label{Bc}
	\hat B(t)
	= e^{t\lambda_-} \hat B^{0}
	+\left(e^{t\lambda_+}-e^{t\lambda_-}\right) b^{0}
	+ \frac{i}{\lambda_-}
	\int_0^t\left( e^{(t-\tau)\lambda_+}
	- e^{(t-\tau)\lambda_-} \right) \xi\times e
	\;d\tau,
\end{equation}
where, 
$$
		\lambda_\pm (\xi)=
		\frac {-1 \pm \sqrt{ 1 - 4 |\xi|^2}} {2},
$$
the initial data $\left(\hat E^{0},\hat B^{0}\right)$ and the source term $\nabla\wedge G$ are decomposed as follows
\begin{equation}\label{Duha1}
	\begin{pmatrix}
		\hat E^{0} \\ \hat B^{0}
	\end{pmatrix}
	=
	\begin{pmatrix}
		\frac{\xi\cdot\hat E^{0}}{|\xi|^2}\xi \\ 0
	\end{pmatrix}
	+
	\begin{pmatrix}
		e^{0} \\ \frac{-i}{\lambda_-}\xi\times e^{0}
	\end{pmatrix}
	+
	\begin{pmatrix}
		\frac{-i}{\lambda_-}\xi\times b^{0} \\ b^{0}
	\end{pmatrix},
\end{equation}
with $\xi\cdot e^{0}=\xi\cdot b^{0}=0$, and

\begin{equation}
\label{Duha2}
	\begin{pmatrix}
		\hat G \\ 0
	\end{pmatrix}
	=
	\begin{pmatrix}
		\frac{\xi\cdot\hat G}{|\xi|^2}\xi \\ 0
	\end{pmatrix}
	+
	\begin{pmatrix}
		e \\ \frac{-i}{\lambda_-}\xi\times e
	\end{pmatrix}
	+
	\begin{pmatrix}
		\frac{-i}{\lambda_-}\xi\times b \\ b
	\end{pmatrix},
\end{equation}
with $\xi\cdot e=\xi\cdot b=0$.\\ 
Let $K$ be some fixed parameter $1<K<2$ determined such that for all $|\xi|\geq \frac 1{2K}$ we have $\mathcal R(\lambda_\pm)<-\frac14$.
To this end, we further decompose the magnetic fields as follows.
\begin{equation*}
	B =  B_< + B_> ,
\end{equation*}
where, for \begin{equation*}
	\begin{aligned}
		\hat B_< & = \mathds{1}_{\left\{ |\xi|\leq \frac{1}{2K}\right\}} \hat B,
		\\
		\hat B_> & = \mathds{1}_{\left\{\frac{ 1}{2K}<|\xi|\right\}} \hat B.
	\end{aligned}
\end{equation*}
The first corresponds to the low frequency component of $B$ and the second to the high frequency. Now we estimate each of the above terms separately. Thanks to \eqref{integ est}, we have

\begin{eqnarray*}
t^\frac12 2^{\frac k2}\|\Delta_kB_>(t)\|_{L^2}
&\lesssim& t^\frac12 2^\frac k2 e^{-\frac t4}(\|\Delta_kE^0\|_{L^2}+\|\Delta_kB^0\|_{L^2})\\
&+&t^\frac12 \int_0^te^{-\frac{t-s}4}2^\frac k2\|\Delta_k G(s)\|_{L^2}\;ds\\
&\lesssim&
2^\frac k2 (\|\Delta_kE^0\|_{L^2}+\|\Delta_kB^0\|_{L^2})\\
&+&
\sup_{t>0}\Big(t^\frac122^{\frac k2}\|\Delta_k G(t)\|_{L^2}\Big)t^\frac12e^{-\frac t4}\int_0^t\frac{e^{\frac s4}}{\sqrt s}\;ds.
\end{eqnarray*}
Again, using \eqref{integ est} we conclude that 
\begin{eqnarray}
\label{est B1}
t^\frac12 2^{\frac k2}\|\Delta_kB_>(t)\|_{L^2}
&\lesssim& \sup_{t>0}\Big(t^\frac122^{\frac k2}\|\Delta_k G(t)\|_{L^2}\Big)\\
&+&2^\frac k2 (\|\Delta_kE^0\|_{L^2}+\|\Delta_kB^0\|_{L^2}).
\nonumber
\end{eqnarray}
Now we estimate $B_<$. From Duhamel's formula, Lemma \ref{lambda}  and \eqref{integ est} we estimate

\begin{eqnarray*}
t^\frac12 2^k\|\Delta_kB_<(t)\|_{L^2}
&\lesssim& t^\frac12 2^k (e^{-\frac t2}\|\Delta_k(B^0-b^0)\|_{L^2}+e^{-t2^k}\|\Delta_kb^0\|_{L^2})
\\
&+&
t^\frac122^k\int_0^te^{-(t-s)2^{2k}}\|\Delta_k G_<(s)\|_{L^2}\;ds\\
&+& t^\frac122^k\int_0^te^{-\frac{t-s}2}\|\Delta_k G_<(s)\|_{L^2}\;ds\\
&\lesssim&
2^k\|\Delta_k(B^0-b^0)_<\|_{L^2}+\|\Delta_kb^0_<\|_{L^2}+\tilde{\sup_{t>0}}t^\frac12\|\Delta_k G_<(t)\|_{L^2_x}
\end{eqnarray*}
which, in virtue of Lemma \ref{lambda} and the identity 
\begin{eqnarray}
\label{B b}
\hat B^0-b^0=-\frac i{\lambda_-}\xi\times e^0.
\end{eqnarray}
gives
\begin{eqnarray*}
t^\frac12 2^k\|\Delta_kB_<(t)\|_{L^2}
&\lesssim& \|\Delta_k(B^0-b^0)_<\|_{L^2}+\|\Delta_kb^0_<\|_{L^2}\\
&+&\tilde{\sup_{t>0}}t^\frac12\|\Delta_k G_<(t)\|_{L^2_x}.
\end{eqnarray*}
Taking the $\ell^2$ summation in $k$ gives
$$
\tilde{\sup_{t>0}}\;t^\frac12\|B_<\|_{\dot H^1}\lesssim \|(E^0,B^0)\|_{L^2}+
\tilde{\sup_{t>0}}\;t^\frac12\|G\|_{L^2_x},
$$
as desired. To estimate $E$, it is sufficient to estimate 
$$
\tilde{\sup_{t>0}}\;t^\frac12\|E_\nabla\|_{H^\frac12},\quad\mbox{and}\quad 
\tilde{\sup_{t>0}}\;t^\frac12\|\nabla\times E_\sigma\|_{\dot H^{-1,-\frac12}},
$$
because 
$$
\tilde{\sup_{t>0}}\;t^\frac12\|E_\sigma\|_{H^\frac12}\lesssim \tilde{\sup_{t>0}}\;t^\frac12\|\nabla\times E_\sigma\|_{\dot H^{-1,-\frac12}}.
$$
Thanks to Faraday's law, we have
\begin{eqnarray*}
\|\nabla\times E_\sigma\|_{\dot H^{-1,-\frac12}}&=&\|\partial_t B\|_{\dot H^{-1,-\frac12}}.
\end{eqnarray*}
From Duhamel's formula, we have

\begin{eqnarray*}
\partial_t\hat B=\lambda_-e^{t\lambda_-}\hat B^0+(\lambda_+e^{t\lambda_+}-\lambda_-e^{t\lambda_-})b^0+\frac i{\lambda_-}
\int_0^t(\lambda_+e^{(t-s)\lambda_+}-\lambda_-e^{(t-s)\lambda_-})\xi\times e\;ds.
\end{eqnarray*}
Using Lemma \ref{lambda}, we have for all $k\leq1$, 
\begin{eqnarray*}
t^\frac122^{-k}\|\Delta_k(\nabla\times E_\sigma)\|_{L^2}&\lesssim& t^\frac122^{-k}2^{2k}e^{-t2^{2k}}\|\Delta_kb^0\|_{L^2}
+t^\frac122^{-k}e^{-\frac t2}\|\Delta_k(B^0-b^0)\|_{L^2}\\
&+&t^\frac122^{2k}\int_0^te^{(s-t)2^{2k}}\|\Delta_ke(s)\|_{L^2}\;ds
+t^\frac12\int_0^te^{\frac{s-t}2}\|\Delta_ke(s)\|_{L^2}\;ds\\
&\lesssim& \|\Delta_kb^0\|_{L^2}+(\|\Delta_kE^0\|_{L^2}+\|\Delta_kB^0\|_{L^2})\\
&+&\sup_{t>0}(t^\frac12\|\Delta_kG\|_{L^2})t^\frac122^{2k}e^{-t2^{2k}}\int_0^t\frac{e^{-s2^{2k}}}{\sqrt s} \;ds\\
&+&\sup_{t>0}(t^\frac12\|\Delta_kG\|_{L^2})t^\frac12e^{-\frac t2}\int_0^t\frac{e^{\frac s2}}{\sqrt s}\;ds.
\end{eqnarray*}
Again, it is important to mention that in the second estimate in above we used \eqref{B b}. Now for $k\geq2$

\begin{eqnarray*}
t^\frac122^{-\frac k2}\|\Delta_k(\nabla\times E_\sigma)\|_{L^2}&\lesssim&
t^\frac122^{\frac k2}e^{-\frac t2}(\|\Delta_kE^0\|_{L^2}+\|\Delta_kB^0\|_{L^2})\\
&+&t^\frac122^{\frac k2}\int_0^te^{(s-t)2^{2k}}\|\Delta_kG\|_{L^2}\;ds\\
&\lesssim&2^{\frac k2}(\|\Delta_kE^0\|_{L^2}+\|\Delta_kB^0\|_{L^2})\\
&+&
\sup_{t>0}(t^\frac122^\frac k2\|\Delta_kG\|_{L^2})t^\frac12e^{-t2^{2k}}\int_0^te^{s2^{2k}}\frac{ds}{\sqrt s}.
\end{eqnarray*}
Using \eqref{integ est} and taking the $\ell^2$ summation concludes the proof.
To estimate $\tilde{\sup_{t>0}}\;t^\frac12\|E_\nabla\|_{H^\frac12}$, we also write Duhamel's formula for $E_\nabla$.
$$
E_\nabla=e^{-t}E_\nabla^0+\int_0^te^{s-t} G_\nabla(s)\;ds.
$$
Then, 
\begin{eqnarray*}
t^\frac122^{\frac k2}\|\Delta_k E_\nabla\|_{L^2}&\lesssim& t^\frac122^{\frac k2}e^{-t}\|\Delta_k E_\nabla^0\|_{L^2}
+t^\frac12e^{-t}\sup_{t>0}(t^\frac122^\frac k2\|\Delta_kG_\nabla\|_{L^2})\int_0^t\frac{e^s}{\sqrt s}\;ds
\end{eqnarray*}
which finishes the proof of the Lemma.
\end{proof}


\subsection{Nonlinear estimates}
The following is a series of nonlinear estimates needed for the contraction argument.  The first Lemma is essential to estimate the nonlinearity in Maxwell-Amp\`ere's equation.

\begin{lemma} 
\label{lem-prod1} 
For all smooth functions  $u$, $E$ and $B$ defined on some 
interval $[0,T]$, we have the following estimates, with constants independent of $T$: 

\begin{eqnarray}
\label{est u B} 
\|u\wedge B \|_{\mathcal Y_2}\lesssim 
\|u\|_{\tilde L^2(\dot{\mathcal B}^\frac32_{2,(\infty,1)})\cap \tilde{\rm{sup}_t}t^\frac12 {\dot{\mathcal B}^\frac32_{2,(\infty,1)}} }
\|B\|_{\mathcal X_3}
\end{eqnarray}

\begin{eqnarray}
\label{est uper B} 
\|u_{\text per}\wedge B \|_{\mathcal Y_2}\lesssim   
\|u_{\text per}\|_{\tilde L^2_{\text per}(\dot{\mathcal B}^\frac32_{2,(\infty,1)})\cap  \tilde L^\infty( \dot{\mathcal B}^\frac12_{2,(\infty,1)}   )} 
\|B\|_{\mathcal X_3}
\end{eqnarray}

\begin{eqnarray}
\label{est u Bper} 
\|u\wedge B_{\text per} \|_{\mathcal Y_2}\lesssim   
\|u\|_{\mathcal X} 
\|B_{\text per}\|_{\tilde L^\infty_{\text{per}}(H^\frac12)}
\end{eqnarray}

\end{lemma}
The second Lemma is to estimate bilinear terms in the Navier-Stokes equations 

\begin{lemma}
\label{lem-prod2} 
\begin{eqnarray}
\label{est u u} 
\|\nabla(u\otimes v)\|_{\mathcal Y_1}  \lesssim  
\|u \|_{\mathcal X} \|v \|_{\mathcal X}
\end{eqnarray}

\begin{eqnarray}
\label{est uper u} 
\|\nabla(u_{\text per}\otimes v)\|_{\mathcal Y_1}  \lesssim  
\|v \|_{\mathcal X_1}
\|u_{\text per} \|_{\tilde L^\infty_{\text per}(\dot{\mathcal B}^\frac12_{2,(\infty,1)})}
\end{eqnarray}

\begin{eqnarray}
\label{est Eper B}
\|E_{\text per}\wedge B \|_{\mathcal Y_1}\lesssim
\|E_{\text per}\|_{\tilde L^\infty(H^\frac12)}
\|B\|_{\mathcal X_3},
\end{eqnarray}

\begin{eqnarray}
\label{est E Bper}
\|E\wedge B_{\text per} \|_{\mathcal Y_1}\lesssim
\|E\|_{\mathcal X_2}
\|B_{\text per}\|_{\tilde L^\infty_{\text{per}}(H^\frac12)},
\end{eqnarray}

\begin{eqnarray}
\label{est E B} 
\|E\wedge B \|_{\mathcal Y_1}\lesssim   
\|E\|_{\mathcal X_2} 
\|B\|_{\tilde L^\infty H^\frac12 }
\end{eqnarray}

\end{lemma}

The last Lemma gives the estimates of the trilinear terms in the equation of the velocity vector field.

\begin{lemma}
\label{lem-prod3} 

\begin{eqnarray}
\label{est u Bper Bper} 
\|(u\wedge B_{\text per})\wedge B_{\text per} \|_{\mathcal Y_1}\lesssim   
\|u\|_{\mathcal X} 
\|B_{\text per}\|_{\tilde L^\infty(H^\frac12)}^2
\end{eqnarray}

\begin{eqnarray}
\label{est u B Bper} 
\|(u\wedge B)\wedge B_{\text per} \|_{\mathcal Y_1}\lesssim   
\|u\|_{\mathcal X}\|B\|_{L^\infty(H^\frac12)}\|B_{\text per}\|_{\tilde L^\infty(H^\frac12)}
\end{eqnarray}

\begin{eqnarray}
\label{est uper B B} 
\|(u_{\text per}\wedge B)\wedge B \|_{\mathcal Y_1}\lesssim   
\|u_{\text per}\|_{\tilde L^\infty_{\text per}(\dot{\mathcal B}^\frac12_{2,1})}  
\|B\|_{\mathcal X_3}\|B\|_{\mathcal X_3}
\end{eqnarray}

\begin{eqnarray}
\label{est uper B Bper} 
\|(u_{\text per}\wedge B)\wedge B_{\text per} \|_{\mathcal Y_1}\lesssim   
\|u_{\text per}\|_{\tilde L^\infty_{\text per}(\dot{\mathcal B}^\frac12_{2,(\infty,1)})\cap \tilde L^2_{\text per}(\dot{\mathcal B}^\frac32_{2,(\infty,1)})}  
\|B_{\text per}\|_{\tilde L^\infty(H^\frac12)}\|B\|_{\mathcal X_3}
\end{eqnarray}

\end{lemma}

\begin{proof}
We only show how to estimate the worst terms $(u_{\rm{per}}\wedge B_{\rm{per}})\wedge B$ and $(u\wedge B_{\rm{per}})\wedge B_{\rm{per}}$ because all the other estimates are easier. It is enough to show the following
\begin{eqnarray}
\label{intermediate 1}
\sup_n(1+n)^\frac{1-\varepsilon}2\|u_{\rm{per}}\wedge B\|_{L^2(n,n+1;H^\frac12)}\lesssim \|u_{\rm{per}}\|_{\mathcal X_1}
\|B\|_{\mathcal X_3}
\end{eqnarray}
and 
\begin{eqnarray}
\label{intermediate 2}
\sup_n(1+n)^\frac{1-\varepsilon}2\|B_{\rm{per}}\wedge F\|_{\mathcal Y_1}\lesssim \|B_{\rm{per}}\|_{\tilde L^\infty(H^\frac12)}
\sup_n(1+n)^\frac{1-\varepsilon}2\|F\|_{L^2(n,n+1;H^\frac12)}
\end{eqnarray}
for $F=u\wedge B_{\rm{per}}$.

We begin by proving \eqref{intermediate 1}, and estimate the term $T_{u_{\rm{per}}}B$ in the para-product. First, as in the proof of Lemma \ref{law prod lem Bes}, we have
\begin{eqnarray*}
\sup_n(1+n)^\frac{1-\varepsilon}2\|T_{u_{\rm{per}}}B\|_{L^2(n,n+1;H^\frac12)}&\lesssim& \|u_{\rm{per}}\|_{L^\infty_{\rm{per}}(L^\infty)}
\tilde{\sup_t}(1+t)^\frac{1-\varepsilon}2\|B\|_{H^{\frac12}}\\
&\lesssim&\|u_{\rm{per}}\|_{\mathcal X_1}\tilde{\sup_t}(1+t)^\frac{1-\varepsilon}2\|B\|_{H^{\frac12}}.
\end{eqnarray*}
Next, for $k\leq0$, by Bernstein's lemma we have 
\begin{eqnarray*}
\sum_{j\leq k}2^j\|\Delta_jB\|_{L^2}2^\frac{j-k}22^k\|\Delta_ku_{\rm{per}}\|_{L^2}&\leq &\sum_{j\leq k}2^j\|\Delta_jB\|_{L^2}2^\frac{j-k}2  
\|u_{\rm{per}}\|_{\tilde L^\infty_{\rm{per}}(\dot B^\frac12_{2,(\infty,1)})}
\end{eqnarray*} 
while for $k\geq1$, we decompose further as follows
\begin{eqnarray}
\nonumber
\sum_{j\leq k}2^{\frac{3j}2}\|\Delta_jB\|_{L^2}2^\frac k2\|\Delta_ku\|_{L^2}&\leq &\sum_{j\leq0}2^j\|\Delta_jB\|_{L^2}2^\frac j2  2^\frac k2\|\Delta_ku_{\rm{per}}\|_{L^2}\\
&+&\sum_{0\leq j\leq k}2^\frac j2\|\Delta_jB\|_{L^2}2^{j-k}  2^\frac{3k}2\|\Delta_ku_{\rm{per}}\|_{L^2}
\nonumber
\end{eqnarray}
which thanks to Young's inequality and after integration in time give

\begin{eqnarray*}
\sup_n(1+n)^\frac{1-\varepsilon}2\|T_Bu_{\rm{per}}\|_{L^2(n,n+1;H^\frac12)}    &\lesssim& \|u_{\rm{per}}\|_{\tilde L^\infty_{\rm{per}}(\dot B^\frac12_{2,(\infty,1)})}
\sup_n(1+n)^\frac{1-\varepsilon}2\|B\|_{L^2(n,n+1;\dot H^{1,\frac12})}\\
&+&\|u_{\rm{per}}\|_{\tilde L^\infty_{\rm{per}}(\dot B^\frac32_{2,(\infty,1)})}\tilde{\sup_t}(1+t)^\frac{1-\varepsilon}2\|B\|_{H^{\frac12}}.
\end{eqnarray*}
Thus
$$
\sup_n(1+n)^\frac{1-\varepsilon}2\|T_Bu_{\rm{per}}\|_{L^2(n,n+1;H^\frac12)}\lesssim \|u_{\rm{per}}\|_{\mathcal X_1} \|B\|_{\mathcal X_3}.
$$
Finally we estimate the rest $R(u_{\rm{per}},B)$ only when $k\leq0$ because the case $k\geq1$ is easier. By Bernstein's lemma we have 
\begin{eqnarray*}
2^\frac k2\sum_{j\geq k-2}2^{\frac{3k}2}\|\Delta_jB\|_{L^2}\|\Delta_ju_{\rm{per}}\|_{L^2}&\leq &2^{2k}\left(\sum_{j\leq 0}(\cdot)+\sum_{k-2\leq j\leq0}(\cdot)\right)\\
&\lesssim&
\|u_{\rm{per}}\|_{\dot B^{\frac{1}{2}} _{2,(\infty,1)}} \|B\|_{L^2}2^{2k}\\&+&\sum_{j\geq k-2}2^{\frac{3j}2}\|\Delta_ju_{\rm{per}}\|_{L^2} 2^\frac j2 \|\Delta_jB\|_{L^2}2^{2(k-j)}.
\end{eqnarray*}
This concludes that
$$
\sup_n(1+n)^\frac{1-\varepsilon}2\|R(B,u_{\rm{per}}) \|_{L^2(n,n+1;H^\frac12)}\lesssim \|u_{\rm{per}}\|_{\mathcal X_1} \|B\|_{\mathcal X_3}.
$$
Now we show \eqref{intermediate 2}. For $k\leq0$, by Bernstein's lemma we have 
\begin{eqnarray*}
2^{-k(\frac12+\varepsilon)}\|\Delta_kT_{B_{\rm{per}}}F\|_{L^2}\leq2^\frac k2\|\Delta_kF\|_{L^2}\sum_{j\leq k}2^{j(\frac12-\varepsilon)}\|\Delta_jB_{\rm{per}}\|_{L^2}2^{(j-k)(1+\varepsilon)}.
\end{eqnarray*} 
Additionally, when $k\geq1$, we estimate as follows

\begin{eqnarray*}
2^{-\frac k2}\|\Delta_kT_{B_{\rm{per}}}F\|_{L^2}&\lesssim&2^\frac k2\|\Delta_kF\|_{L^2}\left(\|B_{\rm{per}}\|_{L^2}+\sum_{0\leq j\leq k}2^{\frac j2}\|\Delta_jB_{\rm{per}}\|_{L^2}2^{j-k}\right)\\
&\lesssim&2^\frac k2\|\Delta_kF\|_{L^2}\|B_{\rm{per}}\|_{H^\frac12}.
\end{eqnarray*} 
The last two estimates combined give
$$
\sup_n(1+n)^\frac{1-\varepsilon}2\|T_{B_{\rm{per}}}F) \|_{L^2(n,n+1;\dot B^{-\frac12,-\frac12}_{2,(\infty,1)})}\lesssim 
\sup_n(1+n)^\frac{1-\varepsilon}2\|F \|_{L^2(n,n+1;H^\frac12)}
 \|B_{\rm{per}}\|_{\tilde L^\infty(H^\frac12)}.
$$
Since the force $F$ and the magnetic field $B$ share the same space regularity, then the same analysis in above implies
$$
\sup_n(1+n)^\frac{1-\varepsilon}2\|T_F{B_{\rm{per}}}) \|_{L^2(n,n+1;\dot B^{-\frac12,-\frac12}_{2,(\infty,1)})}\lesssim 
\sup_n(1+n)^\frac{1-\varepsilon}2\|F \|_{L^2(n,n+1;H^\frac12)}
 \|B_{\rm{per}}\|_{\tilde L^\infty(H^\frac12)}.
$$
Finally the remaining term is estimated thanks to the following observation
\begin{eqnarray*}
2^{-k(\frac12+\varepsilon)}
\sum_{j\geq k-2}2^{\frac{3k}2}\|\Delta_jB_{\rm{per}}\|_{L^2}\|\Delta_jF\|_{L^2}&\lesssim&
\sum_{j\geq k-2}2^{j(\frac12-\varepsilon)}\|\Delta_jB_{\rm{per}}\|_{L^2} 2^\frac j2\|\Delta_jF\|_{L^2}2^{(k-j)(1-\varepsilon)}.
\end{eqnarray*} 
Now we move to the term $(u\wedge B_{\rm{per}})\wedge B_{\rm{per}}$, and begin by the recalling that
$$
H^\frac12\times H^\frac12\hookrightarrow B^{-\frac12}_{2,1}(\R^3).
$$
Hence, $F:=B_{\rm{per}} \cdot B_{\rm{per}}$ belongs to $B^{-\frac12}_{2,1}(\R^3)$. Next we use again para-product to estimate $F\wedge u$.
For $k\leq0$, by Bernstein's lemma we have 
\begin{eqnarray}
\label{low u high F1}
2^{-k(\frac12+\varepsilon)}\|\Delta_kT_uF\|_{L^2}\leq2^{-\frac k2}\|\Delta_kF\|_{L^2}\sum_{j\leq k}2^{j(\frac32-\varepsilon)}\|\Delta_ju\|_{L^2}2^{(j-k)\varepsilon},
\end{eqnarray} 
and, similarly, when $k\geq1$ we estimate as follows
\begin{eqnarray*}
2^{-\frac k2}\|\Delta_kT_uF\|_{L^2}&\lesssim&2^{-\frac k2}\|\Delta_kF\|_{L^2}\left(\sum_{ j\leq 0}2^{j(\frac32-\varepsilon)}\|\Delta_ju\|_{L^2}2^{j\varepsilon}+\sum_{0\leq j\leq k}2^{\frac {3j}2}\|\Delta_ju\|_{L^2}\right)
\end{eqnarray*} 
giving

\begin{eqnarray}
\label{low u high F2}
2^{-\frac k2}\|\Delta_kT_uF\|_{L^2(n,n+1;L^2)}
\lesssim2^{-\frac k2}\|\Delta_kF\|_{L^2}\|u\|_{\tilde L^2(n,n+1;\dot B^{\frac32-\varepsilon,\frac32}_{2,(\infty,1)}}.
\end{eqnarray} 
Finally, for $k\leq0$, we have
\begin{eqnarray}
\label{low F high u1}
2^{-k(\frac12+\varepsilon)}\|\Delta_kT_Fu\|_{L^2}\leq2^{k(\frac32-\varepsilon)}\|\Delta_ku\|_{L^2}\sum_{j\leq k}2^{-\frac j2}\|\Delta_jF\|_{L^2}2^{2(j-k)},
\end{eqnarray} 
and, similarly, when $k\geq1$ we have
\begin{eqnarray}
\label{low F high u2}\\
\nonumber
2^{-\frac k2}\|\Delta_kT_uF\|_{L^2}&\lesssim&2^{\frac{3k}2}\|\Delta_ku\|_{L^2}\left(2^{-2k}\|F\|_{\dot B^{-\frac12}_{2,1}}
+\sum_{0\leq j\leq k}2^{-\frac j2}\|\Delta_jF\|_{L^2}2^{2(j-k)}.
\right)
\end{eqnarray} 
Putting together \eqref{low F high u1}, \eqref{low F high u1}, \eqref{low F high u1} and \eqref{low F high u2} finishes the proof of the Lemma.

\end{proof}

\subsection{End of the proof of Theorem \ref{MAIN2}}
\noindent
 \begin{proof}\textsc{of the proposition \ref{contraction}} First, notice that $\Phi\left( - e^{t\mathcal{A}} \Gamma^0_{\rm err} \right) = 0$, while by Lemma~\ref{ghr} and Lemma~\ref{pr1}, we have
 
\begin{equation}
\label{tree}
\left\| e^{t\mathcal{A}} \Gamma^0_{\rm err}  \right\|_{\mathcal{X}} \leq C \left\| \Gamma^0 \right\|_{\dot {\mathcal B}^{\frac12}_{2,(\infty,1)}\times H^{\frac12}
\times H^{\frac12}} \leq C \kappa \delta \leq \frac{\delta}{2}
\end{equation}
for $\kappa$ small enough. On the other hand, in below we will prove  that, if $\Gamma^{(1)}$ and $\Gamma^{(2)}$ belong to $B_\delta$, then under the assumptions of the claim, we have
\begin{equation}
\label{arbre}
\left\| \Phi(\Gamma^{(1)}) - \Phi(\Gamma^{(2)}) \right\|_{\mathcal{X}} \leq \frac{1}{2} \left\| \Gamma^{(1)} - \Gamma^{(2)} \right\|_{\mathcal{X}}.
\end{equation}
Then estimates~\eqref{tree} and~\eqref{arbre} easily yield the claim of proposition \ref{contraction}.

\bigskip

To prove~(\ref{arbre}), let $\Gamma^{(j)}:=(u_j,E_j,B_j)^T\in B_{\delta}$ for $j=1,2$. Write further
$$
e^{t{\mathcal A}}\Gamma^0_{\rm err}+\Gamma^{(j)}(t)=(\bar u_j,\bar E_j,\bar B_j)^T
$$
and set $\Gamma:=\Gamma^{(1)}-\Gamma^{(2)}:=(u,E,B)^T$, and 
$\Phi(\Gamma^{(j)}):=\tilde \Gamma^{(j)} = ({\tilde u}_j,{\tilde E}_j,{\tilde B}_j)^T$ be given 
by \eqref{contraction map}. 
Finally, let $\tilde\Gamma:=\tilde\Gamma^{(1)}-\tilde\Gamma^{(2)}
:=({\tilde u},{\tilde E},{\tilde B})^T$.\\ 
We decompose  ${\tilde u}$ into ${\tilde u}=\tilde u^{(a)}+ \tilde u^{(b)}$, with $\tilde u^{(a)}$ accounting 
for the convection term
\begin{eqnarray}
\nonumber
\tilde u^{(a)}:&=&-\int_0^te^{(t-t')\Delta}\mathbb{P}
\nabla(u_1\otimes  u +  u\otimes u_2+u\otimes  u_{\rm per} +  u_{\rm per}\otimes u)\;dt'\\
\nonumber
&+&\int_0^te^{(t-t')\Delta}{\mathbb{P}}(E_1\wedge B+E\wedge B_2+E\wedge B_{\rm per}+E_{\rm per}\wedge B)\;dt'\\
\nonumber
&+&\int_0^te^{(t-t')\Delta}{\mathbb{P}}((u\wedge B_1)\wedge B_1+(u_2\wedge E)\wedge B_1+(u_2\wedge B_2)\wedge B)\;dt'\\
\nonumber
&+&\int_0^te^{(t-t')\Delta}\mathbb{P}((u\wedge B_{\rm per})\wedge B_{\rm per})\;dt'
\end{eqnarray}
and 
\begin{eqnarray*}
\tilde u^{(b)}:&=&\int_0^te^{(t-t')\Delta}\mathbb{P}((u_{\rm per}\wedge B)\wedge B_{\rm per}+(u_{\rm per}\wedge B_{\rm per})\wedge B)\;dt'\\
\nonumber
\end{eqnarray*}
Moreover, the electro-magnetic field $(\tilde E,\tilde B)$ satisfies
\begin{eqnarray}
\label{max}
\partial_t \tilde E-\nabla\wedge \; \tilde B+\tilde E&=&u_1\wedge B+  u\wedge B_2+u\wedge B_{\rm per}+u_{\rm per}\wedge B\\
\partial_t \tilde B+\nabla\wedge \; \tilde E&=&0\nonumber
\end{eqnarray}
with $0$ initial data. Applying Lemmas \ref{lem-prod1}, \ref{lem-prod2} and \ref{lem-prod3} immediately gives estimates~\eqref{tree} and~\eqref{arbre} and thus finishes the proof of the proposition \ref{contraction}.
\end{proof}

\clearpage
\appendix
\addcontentsline{toc}{part}{Appendix}

\section{Spectral properties of Maxwell's operator}

Here, we detail the linear analysis of Maxwell's system \eqref{Maxwell}, whose spectral decomposition will be essential for the proof of Lemma \ref{pr1}.

Clearly, Maxwell's system \eqref{Maxwell} may be recast as
\begin{equation*}
	\partial_t
	\begin{pmatrix}
		E\\B
	\end{pmatrix}
	=\mathcal{L}
	\begin{pmatrix}
		E\\B
	\end{pmatrix}
	+
	\begin{pmatrix}
		G\\0
	\end{pmatrix},
\end{equation*}
where Maxwell's operator $\mathcal{L}$ is given by
\begin{equation*}
	\mathcal{L}:=
	\begin{pmatrix}
		-\mathrm{Id} & \nabla\wedge \\
		-\nabla\wedge & 0 \\
	\end{pmatrix}.
\end{equation*}
More precisely, the operator
\begin{equation*}
	\mathcal{L}: \mathcal{D}\left(\mathcal{L}\right)\subset X\rightarrow X,
\end{equation*}
is defined as an unbounded linear operator, where
\begin{equation*}
	X:=\{ \left(E,B\right)\in \left(L^2\left(\mathbb{R}^3\right)\right)^2 \quad\mbox{such that}\quad { \text{div} B = 0 }\},
\end{equation*}
whose domain is given by
\begin{equation*}
	\mathcal{D}\left(\mathcal{L}\right):=\{ \left(E,B\right)\in X \quad \left(\mathbb{P}E,B\right)\in \left(H^1\left(\mathbb{R}^3\right)\right)^2 \},
\end{equation*}
where $\mathbb{P}:L^2\left(\mathbb{R}^3\right)\rightarrow L^2\left(\mathbb{R}^3\right)$ denotes the Leray projector over solenoidal vector fields. 

Next, in order to refine our understanding of the action of the semigroup and the ensuing behaviour of the electromagnetic field $(E,B)$, we conduct a spectral analysis of $\mathcal{L}$.
Since, it has constant coefficients, we use the Fourier transform, which is denoted by
\begin{equation*}
	\mathcal{F}f(\xi) = \hat f(\xi) := \int_{\mathbb{R}^d} e^{-i\xi\cdot x} f(x) dx,
\end{equation*}
and its inverse by
\begin{equation*}
	\mathcal{F}^{-1}g(x) = \tilde g(x) := \frac {1}{\left(2\pi\right)^d}\int_{\mathbb{R}^d} e^{i x\cdot \xi} g(\xi) d\xi.
\end{equation*}

For every $\xi\in\mathbb R^3\setminus\{0\}$, we further define the subspace
\begin{equation*}
	\begin{aligned}
		\mathcal{E}(\xi) & :=
		\{\left( e , b \right) \in \mathbb{C}^3\times\mathbb{C}^3\quad\mbox{such that}\quad{ \xi \cdot b =0}\},
	\end{aligned}
\end{equation*}
and the linear finite-dimensional operator $\hat{\mathcal{L}}(\xi):\mathcal{E}(\xi)\rightarrow\mathcal{E}(\xi)$ by
\begin{equation*}
	\hat{\mathcal{L}}(\xi)
	\begin{pmatrix}
		e\\b
	\end{pmatrix}
	:=
	\begin{pmatrix}
		-  e + i \xi\wedge b \\ -i \xi \wedge e
	\end{pmatrix}.
\end{equation*}

Clearly, $\mathcal{E}(\xi)$ is a $5$-dimensional vector subspace of $\mathbb{C}^3\times\mathbb{C}^3$ and any $(E,B)\in X$ satisfies that $\left(\hat E(\xi),\hat B(\xi)\right)\in \mathcal{E}(\xi)$, for almost every $\xi\in\mathbb{R}^3$.
Finally, note that, for any $(E,B)\in X$,
\begin{equation*}
	\mathcal{F}\left(
	\mathcal{L}
	\begin{pmatrix}
		E\\B
	\end{pmatrix}\right)(\xi)
	=
	\begin{pmatrix}
		- \hat E + i \xi\wedge \hat B \\ -i \xi \wedge \hat E
	\end{pmatrix}
	=
	\hat{\mathcal{L}}(\xi)
	\begin{pmatrix}
		\hat E \\ \hat B
	\end{pmatrix},
\end{equation*}
and
\begin{equation*}
	\mathcal{F}\left(
	e^{t\mathcal{L}}
	\begin{pmatrix}
		E\\B
	\end{pmatrix}\right)(\xi)
	=
	e^{t\hat{\mathcal{L}}(\xi)}
	\begin{pmatrix}
		\hat E \\ \hat B
	\end{pmatrix}.
\end{equation*}
Then, we have the following properties proven in \cite{AIM}.

\begin{proposition}\label{pr11}
	For $|\xi|\neq \frac{1}{2}$
	the distinct eigenvalues of $\hat{\mathcal{L}}(\xi)$ are $\lambda_0=-1$, $\lambda_+(\xi)$ and $\lambda_-(\xi)$,
	with
	\begin{equation}\label{roots}
		\lambda_\pm (\xi)=
		\frac {-1 \pm \sqrt{ 1 - 4 |\xi|^2}} {2}.
	\end{equation}
	Furthermore, there exists a basis of eigenvectors (i.e. $\hat{\mathcal{L}}(\xi)$ is diagonalizable)
	and the eigenspaces corresponding to $\lambda_0$, $\lambda_+(\xi)$ and $\lambda_-(\xi)$ are respectively given by
	\begin{equation*}
		\begin{aligned}
			{\mathcal E}_0(\xi) & =\left\langle
			\begin{pmatrix}
				\xi \\ 0
			\end{pmatrix}
			\right\rangle, \\
			{\mathcal E}_+(\xi) & =
			\set{
			\begin{pmatrix}
				e\\ \frac{-i}{\lambda_+}\xi\times e
			\end{pmatrix}\in\mathbb{C}^3\times\mathbb{C}^3
			}{e\in\mathbb{C}^3,\ \xi\cdot e=0} \\
			& = \set{
			\begin{pmatrix}
				\frac{-i}{\lambda_-}\xi\times b \\ b
			\end{pmatrix}\in\mathbb{C}^3\times\mathbb{C}^3
			}{b\in\mathbb{C}^3,\ \xi\cdot b=0,} \\
			{\mathcal E}_-(\xi) & =
			\set{
			\begin{pmatrix}
				e\\ \frac{-i}{\lambda_-}\xi\times e
			\end{pmatrix}\in\mathbb{C}^3\times\mathbb{C}^3
			}{e\in\mathbb{C}^3,\ \xi\cdot e=0} \\
			& = \set{
			\begin{pmatrix}
				\frac{-i}{\lambda_+}\xi\times b \\ b
			\end{pmatrix}\in\mathbb{C}^3\times\mathbb{C}^3
			}{b\in\mathbb{C}^3,\ \xi\cdot b=0 }.
		\end{aligned}
	\end{equation*}
	
	For any $\xi\in\mathbb R^3\setminus\{0\}$, such that $|\xi|= \frac{1}{2}$ ,
	the distinct eigenvalues of ${\mathcal L}(\xi)$ are $\lambda_0=-1$ and $\lambda_1=-\frac{1}{2}$.
	Furthermore, $\hat{\mathcal{L}}(\xi)$ is not diagonalizable
	and the eigenspaces corresponding to $\lambda_0$ and $\lambda_1$ are respectively given by
	\begin{equation*}
		\begin{aligned}
			{\mathcal E}_0(\xi) & =\left\langle
			\begin{pmatrix}
				\xi \\ 0
			\end{pmatrix}
			\right\rangle, \\
			{\mathcal E}_1(\xi) & =
			\set{
			\begin{pmatrix}
				e\\ \frac{2i}{\sigma c}\xi\times e
			\end{pmatrix}\in\mathbb{C}^3\times\mathbb{C}^3
			}{e\in\mathbb{C}^3,\ \xi\cdot e=0\text{ (and $e_3=0$ if $d=2$)} } \\
			& = \set{
			\begin{pmatrix}
				\frac{2i}{\sigma c}\xi\times b \\ b
			\end{pmatrix}\in\mathbb{C}^3\times\mathbb{C}^3
			}{b\in\mathbb{C}^3,\ \xi\cdot b=0\text{ (and $b_1=b_2=0$ if $d=2$)} }.
		\end{aligned}
	\end{equation*}
	The generalized eigenspace corresponding to $\lambda_1$ is given by
	\begin{equation*}
			\mathcal{K}_1(\xi)=\set{
			\begin{pmatrix}
				e\\ b
			\end{pmatrix}\in\mathbb{C}^3\times\mathbb{C}^3
			}{\xi\cdot e=\xi\cdot b=0\text{ (and $e_3=b_1=b_2=0$ if $d=2$)} }.
	\end{equation*}
\end{proposition}

\begin{rem}
	One easily verifies that the operator $\hat{\mathcal{L}}(\xi)$ is not normal (i.e. it does not commute with its adjoint) and, therefore, the eigenspaces $\mathcal{E}_0(\xi)$, $\mathcal{E}_+(\xi)$ and $\mathcal{E}_-(\xi)$ are not all orthogonal to each other. However, it is seen that each $\mathcal{E}_\pm(\xi)$ is orthogonal to $\mathcal{E}_0(\xi)$.
\end{rem}

\begin{lemma}\label{lambda}
	Let $\xi\in\mathbb{R}^3$ and consider the eigenvalues $\lambda_\pm(\xi)$ defined by \eqref{roots}.
	Then, if $|\xi|\leq \frac{1}{2}$,
	\begin{equation*}
		\begin{gathered}
			-1\leq \lambda_- (\xi)\leq -\frac{1}{2} \leq -|\xi|\leq - {2|\xi|^2}\leq\lambda_+ (\xi)\leq - {|\xi|^2}, \\
			\text{and}\qquad
			\sqrt{ 1 - \left({2 |\xi|}\right)^2}\leq
			\frac{\lambda_-(\xi)-\lambda_+(\xi)}{\lambda_-(\xi)}
			\leq 2 \sqrt{ 1 - \left({2 |\xi|}\right)^2},
		\end{gathered}
	\end{equation*}
	and, if $|\xi|\geq \frac{1}{2}$,
	\begin{equation*}
		\begin{gathered}
			\Re \left(\lambda_-(\xi)\right) = \Re \left(\lambda_+(\xi)\right) =  - \frac {1}2,
			\qquad
			\left|\lambda_+(\xi)\right| = \left|\lambda_-(\xi)\right| = |\xi|, \\
			\text{and}\qquad
			\left|\frac{\lambda_-(\xi)-\lambda_+(\xi)}{\lambda_-(\xi)}\right|=
			2\sqrt{1-\left(\frac{1}{2|\xi|}\right)^2}.
		\end{gathered}
	\end{equation*}
\end{lemma}


Thanks to Proposition \ref{pr11}, decomposing the initial data $\left(E^{0c},B^{0c}\right)\in X$ and the source terms 
$G^c\in L^1_{\mathrm{loc}}\left(\mathbb{R}^+;L^2\left(\mathbb{R}^3\right)\right)$ 
using the eigenspaces of $\hat{\mathcal L}(\xi)$, we write for almost every $\xi\in\mathbb{R}^3$,

\begin{equation}\label{Duha}
	\begin{pmatrix}
		\hat E^{0} \\ \hat B^{0}
	\end{pmatrix}
	=
	\begin{pmatrix}
		\frac{\xi\cdot\hat E^{0}}{|\xi|^2}\xi \\ 0
	\end{pmatrix}
	+
	\begin{pmatrix}
		e^{0} \\ \frac{-i}{\lambda_-}\xi\times e^{0}
	\end{pmatrix}
	+
	\begin{pmatrix}
		\frac{-i}{\lambda_-}\xi\times b^{0} \\ b^{0}
	\end{pmatrix},
\end{equation}
where $\xi\cdot e^{0}=\xi\cdot b^{0}=0$, and
\begin{equation}
\label{Duha3}
	\begin{pmatrix}
		\hat G \\ 0
	\end{pmatrix}
	=
	\begin{pmatrix}
		\frac{\xi\cdot\hat G}{|\xi|^2}\xi \\ 0
	\end{pmatrix}
	+
	\begin{pmatrix}
		e \\ \frac{-i}{\lambda_-}\xi\times e
	\end{pmatrix}
	+
	\begin{pmatrix}
		\frac{-i}{\lambda_-}\xi\times b \\ b
	\end{pmatrix},
\end{equation}
where $\xi\cdot e=\xi\cdot b=0$. 

Next, in view of Proposition \ref{pr11}, the semigroup $e^{t\hat{\mathcal{L}}}$ acts on \eqref{Duha1} as
\begin{equation*}
	e^{t\hat{\mathcal{L}}}
	\begin{pmatrix}
		\hat E^{0} \\ \hat B^{0}
	\end{pmatrix}
	=
	e^{- t}
	\begin{pmatrix}
		\frac{\xi\cdot\hat E}{|\xi|^2}\xi \\ 0
	\end{pmatrix}
	+e^{t\lambda_-}
	\begin{pmatrix}
		e^{0} \\ \frac{-i}{\lambda_-}\xi\times e^{0}
	\end{pmatrix}
	+e^{t\lambda_+}
	\begin{pmatrix}
		\frac{-i}{\lambda_-}\xi\times b^{0} \\ b^{0}
	\end{pmatrix},
\end{equation*}
and on \eqref{Duha3} as
\begin{equation*}
	\begin{aligned}
		\int_0^t
		e^{(t-\tau)\hat{\mathcal{L}}}
		\begin{pmatrix}
			\hat G \\ 0
		\end{pmatrix}(\tau)d\tau
		& =\int_0^t
		e^{- (t-\tau)}
		\begin{pmatrix}
			\frac{\xi\cdot\hat G}{|\xi|^2}\xi \\ 0
		\end{pmatrix}d\tau
		\\
		& +\int_0^te^{(t-\tau)\lambda_-}
		\begin{pmatrix}
			e \\ \frac{-i}{\lambda_-}\xi\times e
		\end{pmatrix}d\tau
		\\
		& +\int_0^te^{(t-\tau)\lambda_+}
		\begin{pmatrix}
			\frac{-i}{\lambda_-}\xi\times b \\ b
		\end{pmatrix}
		d\tau.
	\end{aligned}
\end{equation*}
Therefore, Duhamel's formula \eqref{Duha} yields that
\begin{equation*}
	\hat B(t)
	= \frac{-ic}{\lambda_-}e^{t\lambda_-} \xi\times e^{0}
	+e^{t\lambda_+} b^{0}
	+\int_0^t\left(\frac{-i}{\lambda_-} e^{(t-\tau)\lambda_-}
	\xi\times e
	+ e^{(t-\tau)\lambda_+} b\right)
	d\tau.
\end{equation*}
Further substituting
\begin{equation*}
	\begin{aligned}
		\frac{ic}{\lambda_-}\xi\times e^{0} & =b^{0}-\hat B^{0}, \\
		b & = \frac{i}{\lambda_-}\xi\times e,
	\end{aligned}
\end{equation*}
which is deduced from the second components of \eqref{Duha1} and \eqref{Duha3}, we obtain
\begin{equation*}
	\hat B(t)
	= e^{t\lambda_-} \hat B^{0}
	+\left(e^{t\lambda_+}-e^{t\lambda_-}\right) b^{0}
	+ \frac{i}{\lambda_-}
	\int_0^t\left( e^{(t-\tau)\lambda_+}
	- e^{(t-\tau)\lambda_-} \right) \xi\times e
	d\tau.
\end{equation*}

\end{document}